\documentclass[10pt]{article}
\usepackage{amsmath, amssymb, amsfonts, amsxtra, latexsym, amscd, eucal, amsthm}
\usepackage{graphicx}
\usepackage{epsfig}
\usepackage{color}      % use if color is used in text
\usepackage{hyperref}   % use for hypertext links, including those to external documents and URLs
\usepackage{amssymb} 
\usepackage{amsthm,enumerate}
\newtheorem{Def}{Definition}[section]

\newtheorem{Prop}[Def]{Proposition}
\newtheorem{Lem}[Def]{Lemma}
\newtheorem{Thm}[Def]{Theorem}

\theoremstyle{definition}
\newtheorem{Rem}[Def]{Remark}

\newcommand{\mS}{\mathcal{S}}
\newcommand{\bmS}{\bar{\mathcal{S}}}

\newcommand{\bE}{\mathbb{E}}

\newcommand{\sm}{\sigma}

\usepackage[top=2cm, bottom=2cm, left=2.5cm, right=2.5cm]{geometry}

\allowdisplaybreaks
\begin{document}
\title{%Tamed-adaptive Euler-Maruyama approximation for  SDEs with  discontinuous drift and locally H\"older continuous diffusion coefficients
Tamed-adaptive Euler-Maruyama approximation for  SDEs with superlinearly growing and piecewise continuous drift, superlinearly growing and locally H\"older continuous diffusion}

\author{Minh-Thang Do\footnote{Institute of Mathematics, Vietnam Academy of Science and Technology. Address: 18 Hoang Quoc Viet, Cau Giay, Hanoi, Vietnam. Email: dmthang@math.ac.vn}, \quad Hoang-Long Ngo\footnote{(Corresponding Author) Hanoi National University of Education. Address: 136 Xuan Thuy, Cau Giay, Hanoi, Vietnam. Email: ngolong@hnue.edu.vn}, \quad Nhat-An Pho\footnote{Hanoi National University of Education. Address: 136 Xuan Thuy, Cau Giay, Hanoi, Vietnam. Email: anpn@hnue.edu.vn}}
%\date{Hanoi National University of Education \\ 136 Xuan Thuy - Cau Giay - Hanoi - Vietnam}
\maketitle 
	\begin{abstract}
   In this paper, we consider stochastic differential equations whose drift coefficient is superlinearly growing and piece-wise continuous, and whose diffusion coefficient is superlinearly growing and locally H\"older continuous. We first prove the existence and uniqueness of the solution to such stochastic differential equations and then propose a tamed-adaptive Euler-Maruyama approximation scheme.  We study the rate of convergence in the $L^1$-norm of the scheme in both finite and infinite time intervals.   
\end{abstract}

 \textbf{Keywords:} Adaptive Euler-Maruyama; Discontinuous drift; Locally H\"older continuous diffusion; Stochastic differential equations; Tamed Euler-Maruyama; Uniform in time approximation. 

\textbf{Mathematics Subject Classification:} 60H35,  60H10

\section{Introduction}

Stochastic differential equations (SDEs) are used to model a variety of random dynamical phenomena;
therefore, they play a significant role in many fields of science and industry. In many cases, the
solution to many SDEs cannot be computed explicitly and has to be approximated by a numerical scheme.

In this paper, we consider the numerical approximation for  stochastic processes $X=(X_t)_{t \in [0,+\infty)}$ defined by the following stochastic differential equation (SDE),
	\begin{equation}\label{eqn1.1}
		\begin{cases} 
			& d X_t=b(X_t)dt+\sigma(X_t)dW_t, \\
			& X_0=x_0 \in \mathbb R,
		\end{cases}
	\end{equation}
	where  $W= (W_t)_{t\in [0,+\infty)}$ is a standard Brownian motion defined on a filtered probability space  $(\Omega, \mathcal{F},(\mathcal{F}_t), \mathbb P)$ satisfying the usual condition.

For SDEs with Lipschitz continuous coefficients, it is well-known that the explicit Euler-Maruyama approximation scheme converges at the strong rate of order $\frac{1}{2}$ (see \cite{KP, MT}). 
The numerical approximation for SDEs with irregular coefficients has recently been considered extensively. The divergence of the classical Euler-Maruyama scheme when applying for SDEs with super-linear growth coefficients has been pointed out in \cite{HJKa}. Thereafter, many modified Euler-Maruyama schemes have been introduced for SDEs with super-linear growth coefficients, such as the tamed Euler-Maruyama scheme (see \cite{HJ, HJKa, S1, S2}), the truncated Euler-Maruyama scheme (see \cite{Mao2}), the implicit Euler-Maruyama scheme (see \cite{MS}), and the adaptive scheme (see \cite{FG}). 
The convergence of the Euler-Maruyama scheme for SDEs with H\"older continuous diffusion coefficients was first studied in \cite{GR} by using the Yamada-Watanabe approximation technique. This technique has been developed in  \cite{NT16b, NT2017} to study the strong convergence of the tamed Euler-Maruyama scheme for SDEs with super-linear growth coefficients. 
 
There are several approaches to study numerical approximation for SDEs with discontinuous drift coefficient. The first approach is to transform an SDE with discontinuous drift into a new SDE with Lipschitz continuous coefficients. Using this transformation technique, Leobacher and Sz\"olgyenyi  developed an  approximation method for SDEs with piecewise Lipschitz continuous drift and global Lipschitz continuous diffusion (see \cite{LeoSzo, LS18}). Thereafter,  the transformation technique was applied to study a quasi-Milstein scheme for a smaller family of SDEs, i.e. those with coefficients having piecewise continuous Lipschitz derivatives (see \cite{MY22, MY}). 
The second approach is to use a Gaussian bound for the density of the Euler-Maruyama  approximated solution, which was proven in \cite{LM10}. Using this Gaussian bound, in \cite{NT16b}, Ngo and Taguchi showed that for SDEs 
the Euler-Maruyama approximation with uniformly elliptic, bounded, $(1/2 + \alpha)$-H\"older diffusion coefficient and one-sided Lipschitz, bounded variation drift coefficient, the Euler-Maruyama strongly converges in $L^1$-norm at the order of $\alpha \in (0, 1/2]$. In \cite{NT2017}, by introducing a drift removal transformation technique, they can get rid of the one-sided Lipschitz continuous condition of drift for one-dimensional SDEs (see also \cite{BHY19}). 
The third approach is to use adaptive schemes (see \cite{Larisa, NSS}). The basic idea is that the step size of the scheme will be adjusted to be smaller near the discontinuous points of the drift coefficient. 
The fourth approach is to use the regularisation of the noise. In \cite{DG}, by exploiting the regularising effect of the noise, the authors showed that the rate of strong convergence of the Euler-Maruyama is arbitrarily close to $1/2$ when applying for a larger class of non-degenerated SDEs with irregular drift. The fifth approach is to use the stochastic sewing lemma. In \cite{KLe, BDG}, they studied a tamed Euler-Maruyama scheme of strong order $1/2$  for a class of SDEs with an integrable drift coefficient and elliptic regular diffusion coefficient. The approximation for SDEs with superlinearly growing coefficients, locally Lipschitz diffusion, and piece-wise continuous drift, has been studied very recently in \cite{MSY22, HG22}.

So far, all the numerical studies for SDEs with discontinuous drift have only considered the approximation in a finite time interval. For infinite time intervals, Fang and Giles (see \cite{FG}) developed an adaptive Euler-Maruyama scheme for a class of SDEs with a polynomial growth Lipschitz continuous drift, and bounded, globally Lipschitz continuous diffusion. Using the method of Lyapunov functions, they proved that the proposed scheme converges in $L^p$-norm at the rate of order $1/2$. To extend this result for a larger class of SDEs, i.e. those with locally Lipschitz continuous and one-sided Lipschitz drift and locally $(\alpha+1/2)$-H\"older continuous diffusion coefficients, Kieu et. al. (see \cite{LTT}) established a tamed-adaptive Euler-Maruyama approximation scheme,  utilized the Yamada-Watanabe approximation to obtain upper bounds for some $p^{th}$ moments of the approximated solution, and then showed that the scheme converges in $L^1$-norm at the rate of order $\alpha \in (0, 1/2]$.

In this paper, we are interested in the uniformly in time numerical approximation for the SDEs of the form \eqref{eqn1.1}, where $b$ is superlinearly growing and piecewise locally Lipschitz continuous,  and $\sigma$ is superlinerly growing and locally $(\alpha+1/2)$-H\"older continuous. By combining and improving the tamed and adaptive models in the literature, we introduce a tamed-adaptive Euler-Maruyama approximation scheme that converges in both finite and infinite time intervals. The scheme is proved to converge in $L^1$-norm at the rate of order $\alpha$, and, specifically, under some condition on the growths of $b$ and $\sigma$, the scheme is proved to converge in infinite time intervals. The proof is based on an enhancement from the techniques introduced by Kieu et. al. \cite{LTT}, and Yaroslavtseva \cite{Larisa}, and an introduction of a new function $\varphi$ to control the discontinuity of drift coefficient, which is the most novel technical contribution of this work. To the best of our knowledge, this is the first approximation scheme for SDEs with superlinearly growing and piecewise continuous drift, superlinearly growing, and locally H\"older continuous diffusion. %In comparison to the result of \cite{KLe}, this approach also requires the diffusion coefficient to be uniformly elliptic, but it overcomes the requirement of the integrability of the function $b$.

The rest of this paper is structured as follows. In Section \ref{Sec:2}, we first introduce a list of assumptions on the coefficients $b$ and $\sigma$, and state theorems on the existence and uniqueness of solution for some classes of SDEs with superlinearly growing and piecewise continuous drift, superlinearly growing and locally H\"older continuous diffusion. Next, we introduce the tamed-adaptive Euler-Maruyama approximation scheme and state our main results on the rate of convergence of this scheme in both finite and infinite time intervals. All the proofs are deferred to Section \ref{Sec:3}. In Section \ref{Sec:4}, we present a numerical experiment to illustrate the performance of the new scheme for various SDEs. 

	\section{Tame-adaptive Euler-Maruyama approximation}
	\label{Sec:2} 
	\subsection{Existence and uniqueness of solution} 
	\label{sec:eus}
	
		We first introduce some assumptions on the coefficients of equation \eqref{eqn1.1}. 
	\begin{itemize}

	\item[(A1)] There exist constants $\gamma, \eta \in \mathbb R$ and $p_0 \in [2,+\infty)$ such that
    $$xb(x)+\frac{p_0-1}{2}|\sigma(x)|^2\leq \gamma x^2+\eta,\text{ for any } x \in \mathbb R.$$

    \item[(A2)] There exists a constant $L_1$ such that
    $$(x-y)(b(x)-b(y))\leq L_1(x-y)^2,\text{ for any } x,y \in \mathbb R.$$

\item [(A3)] There exist a sequence $\xi_1<\xi_2<\ldots<\xi_k$,  
 %such that the function $b$ is local Lipschitz continuous in each interval $(\xi_i,\xi_{i+1})$, i.e, there exist 
 some constants $L_2>0$ and $ l \geq 1$ such that for any  $(x,y)\in \mathcal{S} :=  \cup_{i=1}^{k-1} (\xi_i,\xi_{i+1})^2 \cup (-\infty, \xi_1)^2 \cup (\xi_k, +\infty)^2 $,
	$$|b(x)-b(y)|\leq L_2(1+|x|^l+|y|^l)|x-y|.$$

    \item[(A4)] There exist constants  $m\geq 1, L_3 \geq 0$ and $\alpha \in [0,\frac 12]$ such that
    $$|\sigma(x)-\sigma(y)|\leq L_3(1+|x|^m+|y|^m)|x-y|^{\alpha+1/2}, \text{   for any } x,y \in \mathbb R.$$

    \item[(A5)] The function $\sigma$ satisfies $\sigma(\xi_i)>0$ for all $i=1,\ldots,k$,  where $(\xi_i)$ is the sequence defined in (A3). 
    
    \item [(A6)] There exist constants $L_4,h,\xi'>0$ such that
    $$\sup\limits _{|x-y|\leq h,|x|\geq \xi'}\frac{\sigma(x)}{\sigma(y)}\leq L_4.$$
    
     \item[(A'2)] There exists a constant $L_1$ such that
    $$(x-y)(b(x)-b(y))\leq L_1 (x-y)^2,$$ for any $(x,y) \in (-\infty,\xi_1)^2 \cup  (\xi_k,+\infty)^2$, where $\xi_1, \xi_k$ are defined in Assumption (A3). 

     \item [(A'5)] There exists a positive constant $\kappa$ such that
    $\sigma(x)\geq \kappa$ for any $ x \in \mathbb R.$
\end{itemize}
	
\begin{Rem}\label{rm2.1} It can be seen that
    \begin{itemize} 
    \item Assumption (A3) implies that there exists a constant $L'_2>0$ such that 
	$$|b(x)-b(y)|\leq L'_2+L'_2(1+|x|^l+|y|^l)|x-y|,
	\text{ for any } x,y \in \mathbb R.$$
 \item Assumption (A'5) implies Assumption (A5); Assumption (A2) implies Assumption (A'2). 
 \end{itemize} 
\end{Rem}

\begin{Thm}\label{existence1}
    Suppose that Assumptions (A1),(A'2), (A3),(A4),(A'5) hold, and $p_0\ge (4l+4) \vee (4+4\alpha+4m)$. Moreover, if either (A6) hold or (A'2) hold for $L_2<0$, then the equation \eqref{eqn1.1} has a  unique strong solution.
   
\end{Thm}
\begin{Thm}\label{existence2}
    Suppose that (A1)--(A5) hold for $p_0\geq (4l+4)\vee(4m+4\alpha+4)$. Then the equation \eqref{eqn1.1} has a  unique strong solution.
\end{Thm}
The existence and uniqueness of solution to one dimensional SDEs with super-linear growth coefficients, discontinuous drift and degenerate locally Lipschitz continuous diffusion coefficient have been studied recently in  \cite{MSY22, HG22}. The novel of our Theorems \ref{existence1} and \ref{existence2} is that they can be applied for SDE with locally H\"older continuous coefficient. Our approach is based on a localization technique (see \cite{NL17}).  

\subsection{Tamed-adaptive Euler-Maruyama approximation scheme}

Throughout this paper, we always assume that Assumptions (A1), (A3) and (A4) hold. 
Let  $\Xi = \{\xi_1,\ldots,\xi_k\}$, $d(x, \Xi) = \min \limits_{ 1 \leq i \leq k}|x - \xi_i|$, and 
$\Xi^\varepsilon = \{x \in \mathbb R:  d(x, \Xi)   \leq \varepsilon\}$ for any $\varepsilon > 0$. 
Since $\sigma$ is continuous and $\sigma(\xi_i) \neq 0$, there exist $\mu,\nu>0$ depending only on $\sigma$ such that
$\inf\limits_{|x-\xi_i|\leq \mu}\sigma(x)\geq \nu.$ Let $ \varepsilon_0 = \displaystyle  \mu  \wedge \displaystyle \min_{1 \leq i \leq k-1} (\xi_{i+1} - \xi_i)$.

Let $\Delta_0 $ be a positive constant satisfying 
$$\Delta \log^4(1/\Delta) < \sqrt{\Delta} \log^2 (1/\Delta) <\frac 12 \varepsilon_0,$$
 for all $\Delta \in (0, \Delta_0)$. 
For each $\Delta \in (0,\Delta_0)$, we define the functions $\sigma_\Delta$ and $h_\Delta$ by
	$$\sigma_{\Delta}(x):=\frac{\sigma(x)}{1+\Delta^{1/2}|\sigma(x)|},$$
	and 
	\begin{equation}\label{chooseh}
	    h_{\Delta}(x)=
	    \begin{cases}
	         \dfrac{\Delta}{[1+|b(x)|+|\sigma(x)|+|x|^{l}]^2} & \text{ if } \quad x \in (\Xi^{\varepsilon_1})^c, \\
	        \dfrac{[d(x,\Xi)]^2}{\log^4(1/\Delta)[1+|b(x)|+|\sigma(x)|+|x|^l]^2} & \text{ if }\quad x 
	        \in \Xi^{\varepsilon_1}\setminus \Xi^{\varepsilon_2},\\
	        \dfrac{\Delta^2\log^4(1/\Delta)}{[1+|b(x)|+|\sigma(x)|+|x|^l]^2} & \text{ if } \quad x \in \Xi^{\varepsilon_2},
	    \end{cases}
	\end{equation}
	where $\varepsilon_1:=\sqrt{\Delta}\log^2(1/\Delta)$ and $ \varepsilon_2:=\Delta \log^4(1/\Delta)$.

	The tamed-adaptive Euler-Maruyama scheme is defined as follows. Let  $Y_0=x_0, t_0=0$, and for each $i \geq 0$,
	\begin{equation}\label{scheme}
	\begin{cases}
	    &t_{i+1}=t_i+h_{\Delta}(Y_{t_i}),\\
	    &Y_{t}=Y_{t_i}+b(Y_{t_i})(t-t_i)+\sigma_{\Delta}(Y(t_i))(W_{t}-W_{t_i}), \quad t_i <  t \leq t_{i+1}.
	\end{cases}
	\end{equation}
	Note that if Assumptions (A1), (A3) and (A4) hold, then by following the argument in the proof of Proposition 2.1 in \cite{LTT}, it can be shown that the scheme is well-defined, i.e. $\lim \limits_{i\rightarrow \infty}t_i=+\infty,$ a.s.

	For each $t\geq 0$, we define $\underline{t}:=\max\{t_i,t_i\leq t\}$, which is a stopping time, and $\overline{Y}_t:=Y_{\underline{t}}$. 
	
Let $[\frac{p_0}{2}]$ denote the integer part of  $\frac{p_0}{2}$. The following result shows the convergence in $L^1$-norm of the tamed-adaptive Euler-Maruyama scheme. 
	\begin{Thm}\label{convergence}
    Let Assumptions (A1)--(A5) hold and $[\frac{p_0}{2}] \ge (l+1) \vee (1+2\alpha+2m)$.
	Then there exists a positive constant $C$ which does not depend on $\Delta$ such that 
 %$C=C(x_0,L_1,L_2,L_3,\gamma,\eta,(\xi_i),T)$ such that
	\begin{equation} \label{EXmu-X1}
	\sup\limits_{0\leq t\leq T} \bE \left[  |Y_t - X_t| \right] \le \begin{cases}
	C \Delta^{\alpha} &\text{ if } \ 0<\alpha \leq \dfrac{1}{2}, \\ 
	\dfrac{C}{\log \frac{1}{\Delta}}  &\text{ if } \ \alpha=0.
	\end{cases}
	\end{equation}
	Moreover, if $\gamma$ and $L_1$ are negative, then the constant $C$ does not depend on $T$ either.
\end{Thm}
	
	Assumption (A2) is in fact quite restrictive since it excludes some very simple functions, such as $b(x) = 1_{(0,\infty)}(x)$. In the following, we will study the convergence of the tamed-adaptive scheme \eqref{scheme} under Assumption (A'2),  which only requires that $b$ is one-sided Lipschitz continuous outside the interval $(\xi_1, \xi_k)$.   However, we need to assume that $\sigma$ satisfies Assumption (A'5), which is  stronger than Assumption (A5). 

\begin{Thm}\label{convergence2}
   Suppose that Assumptions (A1),(A'2), (A3),(A4) and (A'5) hold, and $[\frac{p_0}{2}] \ge (l+1) \vee (1+2\alpha+2m)$. 
   \begin{enumerate}
       \item[a)] If (A6) holds, then 
       there exists a positive constant $C$ which does not depend on  $\Delta$ %=C(x_0,L_1,L_2,L_3,\gamma,\eta,\xi_i,T)$ 
 such that
	\begin{equation} \label{EXmu-X2}
	\sup\limits_{0\leq t\leq T} \bE \left[  |Y_t - X_t| \right] \le \begin{cases}
	C \Delta^{\alpha} &\text{ if } \ 0<\alpha \leq \dfrac{1}{2}, \\ 
	\dfrac{C}{\log \frac{1}{\Delta}}  &\text{ if } \ \alpha=0.
	\end{cases}
	\end{equation}
 \item[b)]  If (A'2) holds for some $L_1 < 0$, then the estimate \eqref{EXmu-X2} also holds. Moreover, if we suppose further that  $\gamma < 0$, then the constant $C$ does not depend on $T$ either. 
   \end{enumerate}
   
\end{Thm}

Finally, we  consider the complexity of the scheme \eqref{scheme}. For any $T>0$, let $N_T$ be the number of time-steps required by a path approximation on $[0,T]$. More precisely, we write  
$$N_{T}=1+\sum_{k=1}^{\infty}1_{\{t_k<T\}}.$$

\begin{Thm}\label{cost analysis}
    Suppose that Assumptions (A1),(A3),(A4) and (A5) hold, and $[\frac{p_0}{2}]\geq (l+1)\vee(1 +2\alpha+2m)$.
    Then there exists a constant $C>0$ which does not depend on $\Delta$ such that
    \begin{equation*}
        \mathbb E[N_T]\leq
        \begin{cases}
            CT\Delta^{\alpha/2-7/4}  &\text{if } \quad [\frac{p_0}{2}] \leq  \frac{3}{2\alpha + 1} (1 + 2\alpha + 2m),\\
           CT\Delta^{-1}
           \quad &\text{if } \quad
           [\frac{p_0}{2}] > \frac{3}{2\alpha + 1} (1 + 2\alpha + 2m).
        \end{cases}
    \end{equation*}
    Moreover, if $\gamma <0$, then the constant $C$ does not depend on $T$ either.
\end{Thm} 

\begin{Rem}
If $[\frac{p_0}{2}] > (l+1) \vee \frac{3(1 + 2 \alpha + 2m)}{1 + 2\alpha}$, then the number of time discretizations on the interval $[0, T]$ of the scheme \eqref{scheme} is proportional to the one of the classical Euler-Maruyama scheme.  
\end{Rem}

\section{Proofs}
\label{Sec:3} 
\subsection{Yamada-Watanabe approximation}
We recall the Yamada-Watanabe approximation (\cite{Yamada}). For each $\delta>1$ and $\varepsilon>0$, there exists a continuous function $\psi_{\delta\varepsilon}:\mathbb{R}\to\mathbb{R}^+$ with supp$\psi_{\delta\varepsilon}\subset[\varepsilon/\delta;\varepsilon]$ such that
$\int_{\varepsilon/\delta}^\varepsilon \psi_{\delta\varepsilon}(z)dz=1 \text{ and } 0\leq \psi_{\delta\varepsilon}(z)\leq \dfrac{2}{z\log \delta},\ z>0.$
Define
$\phi_{\delta\varepsilon}(x):=\int_0^{|x|}\int_0^y\psi_{\delta\varepsilon} (z) dzdy,\ x\in\mathbb{R}.$
It is easy to verify that $\phi_{\delta \varepsilon}$ has the following useful properties: for any $x\in\mathbb{R} \backslash \{0\}$,
\begin{enumerate}
	\item[(YW1)] $\phi'_{\delta\varepsilon}(x)=\dfrac{x}{|x|}\phi'_{\delta\varepsilon}\left(|x|\right)$,
	\item[(YW2)]
	$0\leq \left|\phi'_{\delta\varepsilon}(x)\right|\leq 1$,
	\item[(YW3)]
		$|x|\leq \varepsilon+\phi_{\delta\varepsilon}(x)$,
	\item[(YW4)] $\dfrac{\phi'_{\delta\varepsilon}(|x|)}{|x|}\leq \dfrac{\delta}{\varepsilon}$,
	\item[(YW5)] $\phi''_{\delta\varepsilon}\left(|x|\right)=\psi_{\delta\varepsilon}(|x|)\leq\dfrac{2}{|x| \log\delta}1_{\left[\frac{\varepsilon}{\delta};\varepsilon\right]}(|x|)\leq \dfrac{2\delta}{\varepsilon\log \delta}$.
\end{enumerate}

\subsection{Some moment estimates}
The following  result has been proven in \cite{LTT}.
\begin{Prop}\label{nghiem dung 1}
	Suppose that  the coefficients $b$ and $\sm$ satisfy Assumption (A1), and $\sigma$ is bounded on every compact subset of  $\mathbb R$.  \textcolor{black}{Then, for any $p \in [0,p_0]$},
	\begin{align*}
	\bE \left[|X_t|^p\right] \le   \left| x_0^2 e^{2\gamma t} + \frac{\eta}{\gamma}(e^{2\gamma t} -1)\right|^{p/2}.
	\end{align*}
\end{Prop}

\begin{Prop}\label{dinh ly 2}
	Suppose that the coefficients $b$ and $\sm$ satisfy Assumptions (A1), (A3) and (A4). Then
	for any positive number  $k \leq  [p_0/2]$,  there exists a positive constant $K=K(x_0,k,\eta,\gamma,L_1,L_2,L_3)$ which does not depend on  $\Delta$, such that
	\begin{align}
	 \bE \left[|Y_t|^{2k}\right] \le \left\{ \begin{array}{l l}
	K e^{2k\gamma t}  &\text{ if } \gamma >0,\\
	K(1+t)^{k}  &\text{ if } \gamma =0,\\
	K  &\text{ if } \gamma <0.
	\end{array} \right. \label{EXtmu}
	\end{align}
\end{Prop}
\begin{proof}
Thanks to H\"older's inequality, it is sufficient to show \eqref{EXtmu} for $0< k \leq p_0/2$.
From the definition of the scheme, we have
\begin{equation*}
\max \left \{ |\overline{Y}_sb(\overline{Y}_s)(s-\underline{s})|, \ b^2(\overline{Y}_s)(s-\underline{s})^2, \  \mathbb{E}\left[\sigma^2_\Delta(\overline{Y}_s)(W_s-W_{\underline{s}})^2\big|\mathcal{F}_{\underline{s}}\right] \right\}  \leq C\Delta.
\end{equation*} 
Then the proof follows directly from the argument of Theorem 2.4 in \cite{LTT}.
\end{proof}

\begin{Lem}\label{hq2}
	Suppose that the coefficients $b$ and $\sm$ satisfy all conditions of Proposition \ref{dinh ly 2}. Then for any $p >0$, there exists a positive constant $C_p$ depending only on $p$ such that
	\begin{equation*}
		\sup_{t \ge 0}\bE \left[  |Y_t - \overline{Y}_t|^p \right] \le C_p\Delta^{p/2}.
	\end{equation*}
\end{Lem}
\begin{proof}
	From \eqref{scheme}, 
	\begin{align}
		|Y_t - \overline{Y}_t|^p 
		&= \left| b(\overline{Y}_t)(t-\underline{t})+\sm_{\Delta}(\overline{Y}_t)(W_t-W_{\underline{t}}) \right|^p \notag \\
		&\le 2^{p-1} \left( \left| b(\overline{Y}_t)(t-\underline{t}) \right|^p + \left| \sm_{\Delta}(\overline{Y}_t)(W_t-W_{\underline{t}}) \right|^p \right) \notag \\
		&\le 2^{p-1} \left( \left| b(\overline{Y}_t)\right|^p \left|h_\Delta(\overline{Y}_t)\right|^p + \left|\sm_{\Delta}(\overline{Y}_t)\right|^p \left|W_t-W_{\underline{t}} \right|^p \right). \notag 
	\end{align} 
By \eqref{chooseh}, we have
			$\left|b(\overline{Y}_t)h_{\Delta}(\overline{Y}_t)  \right| \le \dfrac{\Delta}{4},$ and  $\left|\sigma_{\Delta}(\overline{Y}_t) \left|h_{\Delta}(\overline{Y}_t)\right|^{1/2}  \right| \le  \Delta^{1/2},$
	which implies the desired result.  
\end{proof} 

Lemmas \ref{localtime} and \ref{xac suat chuyen} are modified versions of Lemmas 6 and 7 in \cite{Larisa} for the stochastic differential equations with H\"older continuous diffusion coefficient. 
\begin{Lem}\label{localtime}
Suppose that Assumptions (A1),(A3),(A4) and (A5) hold for $2[\frac{p_0}{2}]\geq (l+1)\vee(4m+4\alpha+2)$.
Let $f:[0,+\infty)\rightarrow [0,+\infty)$ be a Borel measurable function. Let $n = \frac{[\frac{p_0}{2}]}{1 + 2\alpha + 2m}$ and $q=\frac{4n-3}{4n}$. Then there exists a  constant $c$ which does not depend on $\Delta$ such that for any $h\geq 1 ,\varepsilon<\varepsilon_0,\Delta<\Delta_0,$
\begin{equation*}
   \mathbb E\Big[\displaystyle \int_{t}^{t+h} f(d(Y_s,\Xi))1_{\Xi^{\varepsilon}}(Y_s)ds\Big] \leq
    \begin{cases}
    ch\displaystyle\int_{0}^{\varepsilon}f(x)dx+ch\Delta^{\frac 14 +\frac{\alpha}{2}}\sup\limits_{x\in[0,\varepsilon]}f(x),\quad &\text{if }  [\frac{p_0}{2}] \leq  \frac{3(1 + 2\alpha + 2m)}{2\alpha + 1},\\
    ch\displaystyle\int_{0}^{\varepsilon}f(x)dx+ch\sup_{x \in [0,\varepsilon]}f(x)\big(\varepsilon^{q+\frac{\alpha}{2}+\frac 14}+\Delta\big) &\text{if } [\frac{p_0}{2}] >  \frac{3(1 + 2\alpha + 2m)}{2\alpha + 1}.
    \end{cases}
\end{equation*}
Moreover, if $\gamma<0$, then the constant $c$ does not depend on $t$. 
\end{Lem}

\begin{proof}
It is enough to show that for all $i=1,\ldots,k$, there exists $c \in (0,+\infty)$ such that for all $\Delta<\Delta_0,\varepsilon<\varepsilon_0, h\geq 1$,
\begin{equation}\notag
\mathbb E\Big[\displaystyle \int_{t}^{t+h} f(|Y_s-\xi_i|)1_{[\xi_i-\varepsilon,\xi_i+\varepsilon]}(Y_s)ds\Big] \leq \begin{cases}
    ch\displaystyle\int_{0}^{\varepsilon}f(x)dx+ch\Delta^{\frac 14 +\frac{\alpha}{2}}\sup\limits_{x\in[0,\varepsilon]}f(x),\quad &\text{if }  [\frac{p_0}{2}] \leq  \frac{3(1 + 2\alpha + 2m)}{2\alpha + 1}, \\
    ch\displaystyle\int_{0}^{\varepsilon}f(x)dx+ch\sup_{x \in [0,\varepsilon]}f(x)\big(\varepsilon^{q+\frac{\alpha}{2}+\frac 14}+\Delta\big) &\text{if } [\frac{p_0}{2}] >  \frac{3(1 + 2\alpha + 2m)}{2\alpha + 1}.
    \end{cases}
\end{equation}
For each $a\in \mathbb R$, let $L^a(Y)=(L^a_t(Y))_{t\in[0,+\infty)}$ be the local time of $Y$ at the point $a$. From \eqref{scheme}, by Tanaka's formula, for any $h\geq 1$, we have
$$|Y_t-a|=|x_0-a|+\displaystyle\int_{0}^t sgn(Y_s-a)b(\overline{Y}_s)ds+\displaystyle \int_0^t sgn(Y_s-a)\sigma_{\Delta}(\overline{Y}_s)dW_s+L_t^a(Y).$$
%and
%$$|Y_{t+h}-a|=|x_0-a|+\displaystyle\int_{0}^{t+h} sgn(Y_s-a)b_{\Delta}(\overline{Y}_s)ds+\displaystyle \int_0^{t+h} sgn(Y_s-a)\sigma_{\Delta}(\overline{Y}_s)dW_s+L_{t+h}^a(Y).$$
Hence
\begin{align*}
|L_{t+h}^a(Y)-L_t^a(Y)| \leq & |Y_{t+h}-Y_t|+\left|\displaystyle\int_{t}^{t+h} sgn(Y_s-a)b(\overline{Y}_s)ds\right|+\left|\displaystyle \int_t^{t+h} sgn(Y_s-a)\sigma_{\Delta}(\overline{Y}_s)dW_s\right|\\
\leq & \left|\displaystyle\int_{t}^{t+h} b(\overline{Y}_s)ds\right|+
\left|\displaystyle \int_t^{t+h} \sigma_{\Delta}(\overline{Y}_s)dW_s\right|\\
&+\left|\displaystyle\int_{t}^{t+h} sgn(Y_s-a)b(\overline{Y}_s)ds\right|+\left|\displaystyle \int_t^{t+h} sgn(Y_s-a)\sigma_{\Delta}(\overline{Y}_s)dW_s\right|.
\end{align*}
For the rest of the proof, we denote by  $K_1,K_2,...$ some constants that do not depend on $\Delta,h$ or $a$. Moreover, when $\gamma<0$, these constants do not depend on $t$ either. 
By taking expectations on both sides of the above inequality and using Doob's inequality, we get
\begin{align*}
\mathbb E\Big[|L_{t+h}^a(Y)-L_t^a(Y)|\Big]
\leq& 2\displaystyle\int_{t}^{t+h} \mathbb E[|b(\overline{Y}_s)|]ds+
2\left|\displaystyle \int_t^{t+h}\mathbb E\Big[ \sigma_{\Delta}^2(\overline{Y}_s)\Big]ds\right|^{1/2}\\
\leq& K_1 \displaystyle\int_{t}^{t+h} \bigg( \mathbb E\Big[(1+|\overline{Y}_s)|^l)|\overline{Y}_s|\Big]+1\bigg)ds\\&+
K_1 \left|\displaystyle \int_t^{t+h}\bigg(\mathbb E\Big[ \big(1+|\overline{Y}_s|^m\big)^2|\overline{Y_s}|^{2\alpha+1}\Big]+1\bigg)ds\right|^{1/2},
\end{align*}
where the last estimate is derived from Assumptions (A3), (A4), Remark \ref{rm2.1}, and the fact that $\sigma_{\Delta}^2(x)\leq \sigma^2(x)$.
Thanks to Proposition \ref{dinh ly 2}, 
$$\sup\limits_{s\in[0,t]}\mathbb E\Big[|\overline{Y}_s|^{l+1}\Big]+ \sup\limits_{s \in [0,t)}\mathbb E\Big[|\overline{Y}_s|^{2m+2\alpha+1}\Big]\leq K_2.$$
%From now in this proof we denote by $c_1,c_2,$
This implies that 
\begin{equation} \label{eqn:L-L} 
    \mathbb E\Big[|L_{t+h}^a(Y)-L_t^a(Y)|\Big] \leq K_3\sqrt{h}+K_3h \leq 2K_3h.
\end{equation}

On the other hand, by using the occupation time formula, for all $\varepsilon < \varepsilon_0$ and $\Delta<\Delta_0,$ 
$$\mathbb E\Big[\displaystyle \int_{0}^{t} f(|Y_s-\xi_i|)1_{[\xi_i-\varepsilon,\xi_i+\varepsilon]}(Y_s)\sigma_{\Delta}^2(\overline{Y}_s)ds\Big]=\displaystyle\int_{\mathbb R}f(|a-\xi_i|)1_{[\xi_i-\varepsilon,\xi_i+\varepsilon]}(a)\mathbb E\big[L_t^a(Y)\big]da.$$
%$$\mathbb E\Big[\displaystyle \int_{0}^{t+h} f(|Y_s-\xi_i|).1_{[\xi_i-\varepsilon,\xi_i+\varepsilon]}(Y_s).\sigma_{\Delta}^2(\overline{Y}_s)ds\Big]=\displaystyle\int_{\mathbb R}f(|a-\xi_i|)1_{[\xi_i-\varepsilon,\xi_i+\varepsilon]}(a)\mathbb E\big[L_{t+h}^a(Y)\big]da.$$
Hence, it follows from \eqref{eqn:L-L} that 
\begin{align} 
\mathbb E\Big[\displaystyle \int_{t}^{t+h} f(|Y_s-\xi_i|)1_{[\xi_i-\varepsilon,\xi_i+\varepsilon]}(Y_s)\sigma_{\Delta}^2(\overline{Y}_s)ds\Big]=&\displaystyle\int_{\mathbb R}f(|a-\xi_i|)1_{[\xi_i-\varepsilon,\xi_i+\varepsilon]}(a)\Big(\mathbb E\big[L_{t+h}^a(Y)\big]-E\big[L_{t}^a(Y)\big]\Big)da \notag \\
\leq&K_4h\displaystyle\int_{0}^{\varepsilon}f(x)dx. \label{lem3.5:est1}
\end{align}

Next, it follows from Assumption (A4), Lemma \ref{hq2}, Proposition \ref{dinh ly 2}, $2[\frac{p_0}{2}] \geq (4mn+4n\alpha+2n)$ and the Cauchy-Schwarz inequality that 
\begin{align*}
    &\bigg(\mathbb E\Big[\left|\sigma^2(Y_s)-\sigma^2(\overline{Y}_s)\right|^{2n}\Big]\bigg)^{\frac {1}{2n}}\\&\leq L_3^2\bigg(\mathbb E\bigg[\big(1+|Y_s|^m+|\overline{Y}_s|^m\big)^{2n}|Y_s-\overline{Y}_s|^{(\alpha+\frac 12)2n}\big(2\sigma(0)+|Y_s|^{m+\alpha+\frac 12}+|Y_s|^{\frac 12+\alpha}+|\overline{Y}_s|^{m+\alpha+\frac 12}+|\overline{Y}_s|^{\frac 12+\alpha}\big)^{2n}\bigg]\bigg)^{\frac {1}{2n}}\\
&\leq K_5\Delta^{\frac 14 + \frac{\alpha}{2}}.
\end{align*}
Therefore,
\begin{equation}\label{lem3.5:est2}
    \bigg(\mathbb E\Big[|\sigma^2(Y_s)-\sigma^2(\overline{Y}_s)|^\frac{4n}{3}\Big]\bigg)^{\frac {3}{4n}} \leq  \bigg(\mathbb E\Big[|\sigma^2(Y_s)-\sigma^2(\overline{Y}_s)|^{2n}\Big]\bigg)^{\frac {1}{2n}}\leq K_5\Delta^{\frac 14 + \frac{\alpha}{2}}.
\end{equation}
Note that $|\sigma_\Delta^2(x) - \sigma^2(x)| \leq 2|\sigma(x)|^3 \Delta^{\frac 12}$. Thus 
   \begin{align}
   &\bigg( \mathbb E\big[\left|\sigma_{\Delta}^2(\overline{Y}_s)-\sigma^2(\overline{Y}_s)\right|^{\frac{4n}{3}}\big]\bigg)^{\frac {3}{4n}}=\bigg(\mathbb E\Big[\Big(\frac{2\Delta^{1/2}|\sigma^3(\overline{Y}_s)|+\Delta\sigma^4(\overline{Y}_s)}{\big[1+\Delta^{1/2}|\sigma(\overline{Y}_s)|\big]^2}\Big)^{\frac{4n}{3}}\Big]\bigg)^{\frac{3}{4n}}\notag\\ 
    &\leq K_6\Delta^{1/2}\bigg(\mathbb E\Big[|\sigma^{4n}(\overline{Y}_s)|\Big]\bigg)^{\frac{3}{4n}} \leq K_7\Delta^{1/2} \Bigg(\mathbb E\bigg[\sigma^{4n}(0)+L_3^{4n}|\overline{Y}_s|^{4n(\alpha+\frac 12)}\Big(1+|\overline{Y}_s^m|\Big)^{4n}\bigg]\Bigg)^{\frac{3}{4n}}\notag\\
    &\leq K_8\Delta^{1/2}. \label{sigmadelta}
\end{align} 

From \eqref{lem3.5:est1},\eqref{lem3.5:est2} and \eqref{sigmadelta}, for all $\varepsilon \leq \varepsilon_0,$ we have 
\begin{align} 
    &\mathbb E\Big[\displaystyle \int_{t}^{t+h} f(|Y_s-\xi_i|)1_{[\xi_i-\varepsilon,\xi_i+\varepsilon]}(Y_s)ds\Big] \notag \\ \notag
    \leq & \frac{1}{\nu^2}\mathbb E\Big[\displaystyle \int_{t}^{t+h} f(|Y_s-\xi_i|)1_{[\xi_i-\varepsilon,\xi_i+\varepsilon]}(Y_s)\sigma^2(Y_s)ds\Big]\\ \notag
    \leq& \frac{1}{\nu^2}\mathbb E\Big[\displaystyle \int_{t}^{t+h} f(|Y_s-\xi_i|)1_{[\xi_i-\varepsilon,\xi_i+\varepsilon]}(Y_s)\sigma^2_{\Delta}(\overline{Y}_s)ds\Big]\\ \notag
    &+\frac{1}{\nu^2}\mathbb E\Big[\displaystyle \int_{t}^{t+h} f(|Y_s-\xi_i|)1_{[\xi_i-\varepsilon,\xi_i+\varepsilon]}(Y_s)\big[\sigma^2(\overline{Y}_s)-\sigma^2_{\Delta}(\overline{Y}_s)\big]ds\Big]\\ \notag
    &+\frac{1}{\nu^2}\mathbb E\Big[\displaystyle \int_{t}^{t+h} f(|Y_s-\xi_i|)1_{[\xi_i-\varepsilon,\xi_i+\varepsilon]}(Y_s)\big[\sigma^2(Y_s)-\sigma^2(\overline{Y}_s)\big]ds\Big]\\ \notag
    \leq & \frac{K_4h}{\nu^2}\displaystyle\int_{0}^{\varepsilon}f(x)dx+\Big(\frac{K_{8}h^{\frac {3}{4n}}\Delta^{1/2}}{\nu^2}+\frac{K_{5}\Delta^{\frac 14 + \frac{\alpha}{2}}h^{\frac{3}{4n}}}{\nu^2}\Big)\sup\limits_{x \in [0,\varepsilon]}f(x)\Big(\displaystyle\int_{t}^{t+h}\mathbb P(Y_s\in [\xi_i-\varepsilon,\xi_i+\varepsilon])ds\Big)^{\frac {4n-3}{4n}}\\ 
    % \leq & \frac{K_9h}{\nu^2}\displaystyle\int_{0}^{\varepsilon}f(x)dx+\Big(\frac{K_{10}h^{\frac {3}{4n}}\Delta^{1/2}}{\nu^2}+\frac{K_{11}\Delta^{\frac 14 + \frac{\alpha}{2}}h^{\frac{3}{4n}}}{\nu^2}\Big)\sup\limits_{x \in [0,\varepsilon]}f(x)\Big(\displaystyle\int_{t}^{t+h}\mathbb P(Y_s\in [\xi_i-\varepsilon,\xi_i+\varepsilon])ds\Big)^{\frac {4n-3}{4n}}\\ 
    \leq &K_{9}h\displaystyle\int_{0}^{\varepsilon}f(x)dx+K_{10}h^{\frac{3}{4n}}\Delta^{\frac 14 + \frac{\alpha}{2}}\sup\limits_{x \in [0,\varepsilon]}f(x)\Big(\displaystyle\int_{t}^{t+h}\mathbb P(Y_s\in [\xi_i-\varepsilon,\xi_i+\varepsilon])ds\Big)^{\frac {4n-3}{4n}} \label{estimate1}
\end{align}
where the fourth estimate is obtained by using H\"older inequality .
% for $p=\frac{4n}{3}$ and $q=\frac{4n}{4n-3}$.

If $\frac 14+\frac{\alpha}{2}+q<1$, then  $\frac{\alpha}{2}+\frac 14-\frac{3}{4n}<0$. In this case, we get the desired result from \eqref{estimate1} since $\mathbb P(Y_s \in[\xi_i-\varepsilon,\xi_i+\varepsilon])\leq 1$.

It remains to consider the case when $\frac 14+\frac{\alpha}{2}+q\geq 1.$ By choosing $f=1$ in \eqref{estimate1}, we get 
\begin {align*}
&\displaystyle\int_{t}^{t+h}\mathbb P(Y_s\in [\xi_i-\varepsilon,\xi_i+\varepsilon])ds \\
&\leq K_{9}h\varepsilon+K_{10}h^{\frac{3}{4n}}\Delta^{\frac 14 + \frac{\alpha}{2}}\Big(\displaystyle\int_{t}^{t+h}\mathbb P(Y_s\in [\xi_i-\varepsilon,\xi_i+\varepsilon])ds\Big)^{\frac {4n-3}{4n}}\\
&\leq K_{9}h\varepsilon+ K_{10}h\Delta^{\frac 14 + \frac{\alpha}{2}}.
\end{align*}
 By applying the above estimate to  \eqref{estimate1} again,  we get 
\begin{align}
     &\mathbb E\Big[\displaystyle \int_{t}^{t+h} f(|Y_s-\xi_i|)1_{[\xi_i-\varepsilon,\xi_i+\varepsilon]}(Y_s)ds\Big] \notag \\
     \leq &K_{9}h\displaystyle\int_{0}^{\varepsilon}f(x)dx+K_{10}h^{1-q}\sup_{x \in [0,\varepsilon]}f(x)\Delta^{\frac 14 + \frac{\alpha}{2}}\Big(K_{9}h\varepsilon+K_{10}h\Delta^{\frac 14 + \frac{\alpha}{2}}\Big)^q \notag 
\end{align}
By H\"older's inequality and the fact that $\frac 14+\frac{\alpha}{2}+q\geq 1$, we have 
\begin{align*}
\varepsilon^q \Delta^{\frac 14 + \frac{\alpha}{2}} 
%&\leq \frac{q}{q + \frac{\alpha}{2} + \frac 14} (\varepsilon^q)^{\frac{q+ \frac{\alpha}{2}+ \frac 14}{q}} + \frac{\frac{\alpha}{2} + \frac 14}{q+ \frac{\alpha}{2}+ \frac 14}
%(\Delta^{\frac 14 + \frac{\alpha}{2}})^{
%\frac{q+ \frac{\alpha}{2}+ \frac 14}{\frac{\alpha}{2} + \frac14}} %\\
&\leq  
\frac{q}{q + \frac{\alpha}{2} + \frac 14} \varepsilon^{q+ \frac{\alpha}{2}+ \frac 14} 
+ \frac{\frac{\alpha}{2} + \frac 14}{q+ \frac{\alpha}{2}+ \frac 14}
\Delta^{q+ \frac{\alpha}{2}+ \frac 14}\\
& \leq \frac{q}{q + \frac{\alpha}{2} + \frac 14} \varepsilon^{q+ \frac{\alpha}{2}+ \frac 14} 
+ \frac{\frac{\alpha}{2} + \frac 14}{q+ \frac{\alpha}{2}+ \frac 14}
\Delta.
\end{align*}
Therefore, 
\begin{align} 
      &\mathbb E\Big[\displaystyle \int_{t}^{t+h} f(|Y_s-\xi_i|)1_{[\xi_i-\varepsilon,\xi_i+\varepsilon]}(Y_s)ds\Big]\leq &K_{9}h\displaystyle\int_{0}^{\varepsilon}f(x)dx+hK_{11}\sup_{x \in [0,\varepsilon]}f(x)\Big(\varepsilon^{q+\frac{\alpha}{2}+\frac 14}+\Delta^{(\frac 14+\frac{\alpha}{2})(1+q)}+\Delta\Big). \label{ind:k=1}
\end{align}
We will prove by induction that for any $k \in \mathbb N$, there exists a constant $c_k$ such that
\begin{equation}\label{estimate2} 
\mathbb E\Big[\displaystyle \int_{t}^{t+h} f(|Y_s-\xi_i|)1_{[\xi_i-\varepsilon,\xi_i+\varepsilon]}(Y_s)ds\Big] \leq K_{9}h\displaystyle\int_{0}^{\varepsilon}f(x)dx+hc_k\sup_{x \in [0,\varepsilon]}f(x)\Big(\varepsilon^{q+\frac{\alpha}{2}+\frac 14}+\Delta^{(\frac 14+\frac{\alpha}{2})(1+q+\ldots+q^k)}+\Delta\Big).
\end{equation}
Indeed, the estimate \eqref{ind:k=1} verifies \eqref{estimate2} for $k=1$.
Assume that \eqref{estimate2} holds for $k=m>1$. Substituting $f=1$ in \eqref{estimate2} for $k=m$, we have
\begin{align*}
    \displaystyle\int_{t}^{t+h}\mathbb P(Y_s\in [\xi_i-\varepsilon,\xi_i+\varepsilon])ds \leq K_{9}h\varepsilon+hc_m\big(\varepsilon+\Delta+\Delta^{(\frac 12+\frac{\alpha}{4})(1+q+\ldots+q^m)}\big).
\end{align*}
%Choose $f=1$, we obtain
Hence, from \eqref{estimate1} and by using Young's inequality, we obtain
\begin{align*}
   & \mathbb E\Big[\displaystyle \int_{t}^{t+h} f(|Y_s-\xi_i|)1_{[\xi_i-\varepsilon,\xi_i+\varepsilon]}(Y_s)ds\Big]\\
&\leq K_{9}h\displaystyle \int_{0}^{
\varepsilon}f(x)dx+K_{10}h^{1-q}\Delta^{\frac 14 + \frac{\alpha}{2}}\sup_{x \in [0,\varepsilon]}f(x)\Big(\displaystyle\int_{t}^{t+h}\mathbb P(Y_s\in [\xi_i-\varepsilon,\xi_i+\varepsilon])ds\Big)^q\\
&\leq K_{9}h\displaystyle \int_{0}^{
\varepsilon}f(x)dx+K_{10}h\Delta^{\frac 14+\frac{\alpha}{2}}\sup_{x \in [0,\varepsilon]}f(x)\Big[(K_{12}+c_m)\varepsilon+c_m\big(\Delta^{(\frac 14+\frac{\alpha}{2})(1+q+\ldots+q^m)}+\Delta\big)\Big]^q\\
&\leq  K_{9}h\displaystyle \int_{0}^{
\varepsilon}f(x)dx+h c_{m+1}\sup_{x \in [0,\varepsilon]}f(x)\big(\varepsilon^{q+\frac{\alpha}{2}+\frac 14}+\Delta^{(\frac 14+\frac{\alpha}{2})(1+q+\ldots+q^{m+1})}+\Delta\big),
\end{align*}
which implies that \eqref{estimate2} is true for $k=m+1$. Therefore, by induction, \eqref{estimate2} holds for any positive integer $k$. The desired result follows from the fact that $1+q+ q^2+\ldots =\frac{4n}{3}>p$.
\end{proof}

\begin{Lem}\label{xac suat chuyen}
For any $\beta>0$, there exist positive constants  $d_1,d_2$ and $d_3$ such that for all $\Delta \leq \Delta_0$,
\begin{itemize}
    \item [i)] $\sup_{t >0} \mathbb P\Big(|Y_t-Y_{\underline{t}}|\geq \beta\varepsilon_1,Y_{\underline{t}} \in (\Xi^{\varepsilon_1})^c\Big)\leq d_1\Delta^{\beta\log(1/\Delta)}$,
    \item[ii)] $\sup_{t >0} \mathbb P\Big(|Y_t-Y_{\underline{t}}|\geq \beta d(Y_{\underline{t}},\Xi),Y_{\underline{t}} \in \Xi^{\varepsilon_1}\setminus \Xi^{\varepsilon_2}\Big)\leq d_2\Delta^{\beta\log(1/\Delta)},$
    \item[iii)] $ \sup_{t >0} \mathbb P\Big(|Y_t-Y_{\underline{t}}|\geq \beta \varepsilon_2,Y_{\underline{t}} \in \Xi^{\varepsilon_2}\Big)\leq d_3\Delta^{\beta\log(1/\Delta)}.$
\end{itemize}
\end{Lem}
The proof of this Lemma is similar to the one of Lemma 7 in \cite{Larisa}, so it shall not be provided here.

%Next, we define 
%$S:= \Big(\cup_ {i=1}^{k}(\xi_i,\xi_{i+1})^2 \cup (-\infty, \xi_1)^2 \cup (\xi_k, +\infty)^2\Big)^c$.

		\begin{Lem}\label{tich phan xs chuyen}
		For each $t>0$, there exists a positive constant $C$ which does not depend on $\Delta$ such that for all $h\geq 1,\Delta \leq \Delta_0,$
		$$\mathbb E\bigg[\displaystyle\int_{t}^{t+h}1_{ \mS}(Y_s,\overline{Y}_s)ds\bigg]\leq Ch\Delta^{\frac 14 + \frac{\alpha}{2}}.$$
		If $\gamma < 0$, then the constant $C$ does not depend on $t$ either. 
		\end{Lem}
		
		\begin{proof}
		We write 
	\begin{align*}
		&\mathbb E\bigg[\displaystyle\int_{t}^{t+h}1_{ \mS}(Y_s,\overline{Y}_s)ds\bigg]\\
		&=\mathbb E\bigg[\displaystyle\int_{t}^{t+h}1_{ \mS}(Y_s,\overline{Y}_s)1_{(\Xi^{\varepsilon_1})^c}(Y_{\underline{s}})ds\bigg]+\mathbb E\bigg[\displaystyle\int_{t}^{t+h}1_{ \mS}(Y_s,\overline{Y}_s)1_{\Xi^{\varepsilon_1}\setminus \Xi^{\varepsilon_2}}(Y_{\underline{s}})ds\bigg]+\mathbb E\bigg[\displaystyle\int_{t}^{t+h}1_{ \mS}(Y_s,\overline{Y}_s)1_{\Xi^{\varepsilon_2}}(Y_{\underline{s}})ds\bigg]\\
		&  \leq \displaystyle\int_{t}^{t+h}\mathbb P\Big(|Y_s-Y_{\underline{s}}|\geq \varepsilon_1,Y_{\underline{s}}\in (\Xi^{\varepsilon_1})^c\Big)ds
		+ \displaystyle\int_{t}^{t+h}\mathbb P\Big(|Y_s-Y_{\underline{s}}|\geq d(Y_{\underline{s}},\Xi),Y_{\underline{s}}\in \Xi^{\varepsilon_1}\setminus\Xi^{\varepsilon_2}\Big)ds\\
		&\qquad + \displaystyle\int_{t}^{t+h} \mathbb P\big(Y_{\underline{s}}\in \Xi^{\varepsilon_2}\big)ds.
	\end{align*}	
	By Lemma \ref{xac suat chuyen}, the first and second terms of the last expression are bounded by $d_1 h \Delta$ and $d_2 h \Delta$, respectively.  For the last term, it can be bounded by $H_1 + H_2$, where 
	\begin{align*}
	    H_1=  \displaystyle\int_{t}^{t+h} \mathbb P\big(|Y_s-Y_{\underline{s}}|\geq \varepsilon_2,Y_{\underline{s}}\in \Xi^{\varepsilon_2}\big)ds, \quad H_2=  \displaystyle\int_{t}^{t+h} \mathbb P\big(Y_s \in \Xi^{2\varepsilon_2}\big)ds.
	\end{align*}
	Note that 
	$H_1\leq d_3\sqrt{h}\Delta$. 
	By using Lemma \ref{localtime} with $f=1,\varepsilon=\varepsilon_2$, we have
	$H_2\leq 4c\sqrt{h}\varepsilon_2+ch\Delta^{\frac 14 + \frac{\alpha}{2}}$, where $c$ is a constant not depending on $\Delta$. Moreover, if $\gamma <0$, then $c$ does not depend on $t$ either. This concludes the proof.
	\end{proof}
	\begin{Lem}\label{tichphanxacsuatchuyen2}
For each $t >0$ and $L_0 \in \mathbb{R}$, there exists a positive constant $K$ which does not depend on $\Delta$ such that for all $\Delta<\Delta_0$, it holds that 
	\begin{equation}\label{tag30}
	\mathbb E\bigg[\displaystyle\int_{0}^{t}e^{-L_0s} 1_{ \mS}(Y_s,\overline{Y}_s)ds\bigg]\leq K e^{-L_0t}\Delta^{\frac 14 + \frac{\alpha}{2}}.
	\end{equation}
When both $\gamma$ and $L_0$ are negative, the constant $K$ does not depend on $t$ either. 
	\end{Lem}
	\begin{proof}
	It is sufficient to prove \eqref{tag30}  for $t\in \mathbb N$. For $C$ being the constant in the statement of Lemma \ref{tich phan xs chuyen}, we have 
	\begin{align*}
	    &\mathbb E\bigg[\displaystyle\int_{0}^{t}e^{-L_0s}1_{ \mS}(Y_s,\overline{Y}_s)ds\bigg] =\sum\limits_{i=0}^{t-1} \mathbb E\bigg[\displaystyle\int_{i}^{i+1}e^{-L_0s}1_{ \mS}(Y_s,\overline{Y}_s)ds\bigg]\\
	     & \leq e^{|L_0|} \sum\limits_{i=0}^{t-1} e^{-L_0(i+1)} \mathbb E\bigg[\displaystyle\int_{i}^{i+1} 1_{ \mS}(Y_s,\overline{Y}_s)ds\bigg]  \leq C e^{|L_0|} \sum\limits_{i=0}^{t-1} e^{-L_0(i+1)}\Delta^{\frac 14 + \frac{\alpha}{2}} = & C e^{|L_0|} \frac{e^{-L_0(t+1)}-e^{-L_1}}{e^{-L_0}-1}\Delta^{\frac 14 + \frac{\alpha}{2}},
	    	\end{align*}
which implies the desired result. 
	\end{proof}

\subsection{Proof of Theorem \ref{convergence}}

Put $Z_t:=X_t-Y_t$. Applying Property (YW3) and It\^o's formula to $e^{-L_1t}\phi_{\delta \varepsilon}(Y_t)$ gives
		\begin{align}
		&e^{-L_1t}|Z_t| \le e^{-L_1t}\varepsilon + e^{-L_1t}\phi_{\delta \varepsilon}(Z_t)  \notag \\
		=&e^{-L_1t}\varepsilon +\int_0^t  e^{-L_1s}\left[-L_1\phi_{\delta\varepsilon}(Z_s)+\phi'_{\delta\varepsilon}(Z_s)\left( b(X_s)-b(\overline{Y}_s)\right)  +\dfrac{1}{2}\phi''_{\delta\varepsilon}(Z_s)\left| \sm(X_s)-\sm_{\Delta}(\overline{Y}_s) \right|^2 \right]ds  \notag\\ 
		&+ \int_0^t e^{-L_1s}\phi'_{\delta\varepsilon}(Z_s) \left( \sm(X_s)-\sm_{\Delta}(\overline{Y}_s)\right)dW_s.
		 \label{tag25} 
		\end{align}
		Set 
		$J_1(s) = \phi'_{\delta\varepsilon}(Z_s)\left( b(X_s)-b(\overline{Y}_s)\right)$ 
		and 
		$J_2(s)= \dfrac{1}{2}\phi''_{\delta\varepsilon}(Z_s)\left| \sm(X_s)-\sm_{\Delta}(\overline{Y}_s) \right|^2$. 
		Firstly, we write 
		\begin{align}
		J_2(s) 
		&= \dfrac{1}{2}\phi''_{\delta\varepsilon}(Z_s)\left| \sm(X_s)-\sm(Y_s) + \sm(Y_s) - \sm(\overline{Y}_s) + \sm(\overline{Y}_s) -  \sm_{\Delta}(\overline{Y}_s) \right|^2. \notag 
		\end{align}
Using Property (YW5) and Assumption (A4), we have
		\begin{align}
		J_2(s) &\le \frac{3}{|Z_s|\log \delta}1_{[\frac{\varepsilon}{\delta};\varepsilon]}(|Z_s|) \left( \left|\sm(X_s)-\sm(Y_s) \right|^2+\left| \sm(Y_s) - \sm(\overline{Y}_s) \right|^2+\left|\sm(\overline{Y}_s) -  \sm_{\Delta}(\overline{Y}_s) \right|^2 \right) \notag \\
		&\le \frac{3}{|Z_s|\log \delta}1_{[\frac{\varepsilon}{\delta};\varepsilon]}(|Z_s|) \left( L^2_3 \left(1+|{X}_s|^m+|Y_s|^m\right)^2|X_s-Y_s|^{1+2\alpha} + \right. \notag \\ 
		&\quad \left. +L^2_3 \left(1+|Y_s|^m+|\overline{Y}_s|^m\right)^2|Y_s-\overline{Y}_s|^{1+2\alpha} + \Delta\left|\sm(\overline{Y}_s)\right|^4 \right) \notag \\
		&\le \frac{3}{|Z_s|\log \delta}1_{[\frac{\varepsilon}{\delta};\varepsilon]}(|Z_s|) \left( 3L^2_3 \left(1+|{X}_s|^{2m}+|Y_s|^{2m}\right)|X_s-Y_s|^{1+2\alpha} + \right. \notag \\ 
		&\quad \left. + 3L^2_3 \left(1+|Y_s|^{2m}+|\overline{Y}_s|^{2m}\right)|Y_s-\overline{Y}_s|^{1+2\alpha} + \Delta\left|\sm(\overline{Y}_s)\right|^4  \right) \notag \\
		&\le \dfrac{9 L^2_3 \varepsilon^{2\alpha}}{\log \delta}\left(1+|{X}_s|^{2m}+|Y_s|^{2m}\right) +\dfrac{9 L^2_3 \delta}{\varepsilon \log \delta}\left(1+|Y_s|^{2m}+|\overline{Y}_s|^{2m}\right)|Y_s-\overline{Y}_s|^{2\alpha+1} \notag \\
		&\quad +\dfrac{3\delta \Delta |\sm(\overline{Y}_s)|^4}{\varepsilon \log \delta}. \notag
		\end{align} 
		 Using the fact  that $\left(1+|Y_s|^{2m}+|\overline{Y}_s|^{2m}\right)|Y_s-\overline{Y}_s|^{2\alpha+1} \leq \left(1+|Y_s|^{2m}+|\overline{Y}_s|^{2m}\right)^2 \Delta^{\frac 12 + \alpha} + |Y_s-\overline{Y}_s|^{4\alpha+2}\Delta^{-\frac 12 -\alpha}$, we have 
		\begin{align} 
	    J_2(s)	&\le \dfrac{9 L^2_3 \varepsilon^{2\alpha}}{\log \delta}\left(1+|{X}_s|^{2m}+|Y_s|^{2m}\right) +\dfrac{9L^2_3 \delta}{2\varepsilon \log \delta}\left(1+|Y_s|^{2m}+|\overline{Y}_s|^{2m}\right)^2\Delta^{1/2+\alpha} \notag \\
		&\quad +\dfrac{9L^2_3 \delta}{2\varepsilon \log \delta}\Delta^{-\frac 12 -\alpha}|Y_s-\overline{Y}_s|^{4\alpha+2}  +\dfrac{C_1 \delta \Delta \left(|\overline{Y}_s|^{2+4\alpha+4m}+1\right)}{\varepsilon \log \delta} \notag \\
		&\le \dfrac{9 L^2_3 \varepsilon^{2\alpha}}{\log \delta}\left(1+|{X}_s|^{2m}+|Y_s|^{2m}\right) +\dfrac{27L^2_3 \delta}{2\varepsilon \log \delta}\left(1+|Y_s|^{4m}+|\overline{Y}_s|^{4m}\right)\Delta^{1/2+\alpha} \notag \\
		&\quad +\dfrac{9L^2_3 \delta}{2\varepsilon \log \delta}\Delta^{-\frac 12-\alpha}|Y_s-\overline{Y}_s|^{2+4\alpha}  +\dfrac{C_1  \delta \Delta \left( |\overline{Y}_s|^{2+4\alpha+4m}+1\right)}{\varepsilon \log \delta}, \label{tag15}
		\end{align}
		for some constant $C_1>0$.
		Secondly, we write 
		\begin{align}\notag
		J_1(s) 
		&=   \phi'_{\delta\varepsilon}(Z_s)\left( b(X_s)-b(Y_s)\right) + \phi'_{\delta\varepsilon}(Z_s)\left( b(Y_s)-b(\overline{Y}_s)\right).
		\end{align}
	Thanks to Properties (YW1), (YW2) and Assumptions (A2), (A3), we have
		\begin{align*}
		J_1(s) \leq  & \dfrac{\phi'_{\delta \varepsilon}(|Z_s|)}{|Z_s|} Z_s \left( b(X_s)-b(Y_s)\right)  + \left| \phi'_{\delta\varepsilon}(Z_s)\left( b(Y_s)-b(\overline{Y}_s)\right) \right|  \\ \notag
		\le &L_1 \phi'_{\delta \varepsilon}(|Z_s|) |Z_s|+|b(Y_s)-b(\overline{Y}_s)|1_{ \mS}(Y_s,\overline{Y}_s)+|b(Y_s)-b(\overline{Y}_s)|1_{\mS^c}(Y_s,\overline{Y}_s)\\ \notag
		\leq & L_1 \phi'_{\delta \varepsilon}(|Z_s|) |Z_s|+|b(Y_s)-b(\overline{Y}_s)|1_{ \mS}(Y_s,\overline{Y}_s)+L_2(1+|Y_s|^l+|\overline{Y}_s|^l)|Y_s-\overline{Y}_s|. \notag
		\end{align*}
		It follows from Remark \ref{rm2.1} that
		\begin{align}
		J_1(s) &\leq L_1 \phi'_{\delta \varepsilon}(|Z_s|) |Z_s|+C_2 1_{ \mS}(Y_s,\overline{Y}_s)+C_2(1+|Y_s|^l+|\overline{Y}_s|^l)|Y_s-\overline{Y}_s|\\ \notag
		&\leq  L_1 \phi'_{\delta \varepsilon}(|Z_s|) |Z_s|+C_2 1_{ \mS}(Y_s,\overline{Y}_s) +\dfrac{3}{2}C_2 \Delta^{1/2} \left(1+|Y_s|^{2l}+|\overline{Y}_s|^{2l}\right) + \dfrac{1}{2} C_2 \Delta^{-1/2} |Y_s-\overline{Y}_s|^2,
	\label{tag14} 
		\end{align}
		where the constant $C_2>0$ depends on $L_2,n,\xi_i,b(\xi_i+),b(\xi_i-)$.

From \eqref{tag25},\eqref{tag15}, and the property $-L_1\phi_{\delta\varepsilon}(x)+L_1\phi'_{\delta\varepsilon}(|x|)|x| \le \max\{L_1 \varepsilon;0\}$,
\begin{equation*}
\begin{aligned}
		&\bE \left[e^{-L_1t}|Z_t| \right]  \\
		\le \ &e^{-L_1t}\varepsilon + \int_0^t e^{-L_1s}\left[\max\{L_1 \varepsilon;0\}+\dfrac{3}{2}C_2 \Delta^{1/2} \left(1+\bE \left[|Y_s|^{2l}\right]+\bE \left[|\overline{Y}_s|^{2l}\right]\right) \right. \\
		+ &\dfrac{1}{2} C_2 \Delta^{-1/2} \bE \left[|Y_s-\overline{Y}_s|^2\right] +C_2 1_{S}(Y_s,\overline{Y}_s)+ \dfrac{9 L^2_3 \varepsilon^{2\alpha}}{\log \delta}\left(1+\bE \left[|{X}_s|^{2m}\right]+\bE \left[|Y_s|^{2m}\right]\right)   \\
		+&\dfrac{27L^2_3 \delta}{2\varepsilon \log \delta}\left(1+\bE \left[|Y_s|^{4m}\right]+\bE \left[|\overline{Y}_s|^{4m}\right]\right)\Delta^{1/2+\alpha} +\dfrac{9L^2_3 \delta}{2\varepsilon \log \delta}\Delta^{-1/2-\alpha}\bE \left[|Y_s-\overline{Y}_s|^{2+4\alpha}\right]  \\
		+&\left. \dfrac{C_1  \delta \Delta \left(\bE \left[|\overline{Y}_s|^{2+4\alpha+4m}\right]+1\right)}{\varepsilon \log \delta} \right]ds. 
		\end{aligned}
  \end{equation*}
		Thanks to {the condition} $p_0 \ge (2l+2) \vee (2+4\alpha+4m)$, Proposition \ref{dinh ly 2}, Proposition \ref{nghiem dung 1}, and  Lemma \ref{hq2}, there exists a constant $C>0$, which does not depend on $\Delta$, such that for any $0 \leq t \leq T$, it holds that
		\begin{align*}
		\bE \left[e^{-L_1t}|Z_t| \right]
		&\le e^{-L_1t}\varepsilon + C\left[\varepsilon+\Delta^{\frac 12} +\Delta+\Delta^{\frac 14 + \frac{\alpha}{2}} +\dfrac{ \varepsilon^{2\alpha}}{\log \delta}+\dfrac{\delta \Delta^{1/2+\alpha}}{\varepsilon \log \delta} + \dfrac{\delta \Delta }{\varepsilon \log \delta} \right] \int_{0}^{t} e^{-L_1s}ds. 
		\end{align*}
		
		If $\alpha \in (0,\frac{1}{2}]$, by choosing $\varepsilon=\Delta^{\frac 12}$, $\delta=2$, we obtain
		$\sup_{t\leq T} \bE \left[|Z_t|\right] \le C \Delta^{\alpha}$.
		
		If $\alpha=0$,  by choosing $\varepsilon=\Delta^{\frac 14}$, $\delta=\Delta^{-\frac 14}$, we obtain
			$\sup_{t \leq T} \bE \left[|Z_t|\right] \le \frac{C}{\log \frac{1}{\Delta}}. $
			
   Moreover, if $\gamma, L_1<0$, the constant $C$ does not depend on $T$. 
We conclude the proof of Theorem \ref{convergence}.

\subsection{Proof of Theorem \ref{convergence2}}
\subsubsection{Control drift function} 
In \cite{NT2017}, the authors used the function $\varphi$, which is a solution to the equation $b\varphi' + \frac 12 \sigma^2 \varphi'' = 0$ to handle the discontinuity of the drift coefficient $b$ when it is no longer one-sided Lipschitz. Our assumptions on the boundedness of $b$ and $\sigma$ are not as strict as those in \cite{NT2017}. Moreover, we want to show the convergence of the approximation scheme on the whole interval $(0, +\infty)$. Therefore, we will modify the approach in \cite{NT2017} by introducing a new function $\varphi$ defined as follows. 

First, we consider the following properties on the drift coefficient $b$.
\begin{itemize}
    \item[($P_b1$):] $b(y) \geq 0 $ for all $y > \xi_k$. % 1b 
    \item[($P_b2$):] $b(y) \leq 0 $ for all $y < \xi_1$. %2b
\end{itemize}
 
 If Property ($P_b1$) does not hold, we choose $\xi_{k+1} > \xi_k$ such that $b(\xi_{k+1}) < 0$. Otherwise, we choose $\xi_{k+1} = \xi_k+1$. 

If Property ($P_b2$) does not hold, we choose $\xi_0 < \xi_1$ such that $b(\xi_0) > 0$. Otherwise, we choose $\xi_0 = \xi_1 -1 $.

Note that if $L_1 < 0$ then it follows from Assumption (A'2) that neither Property ($P_b1$) nor Property ($P_b2$) holds, which implies that $b(\xi_0)> 0 > b(\xi_{k+1})$.
Also, note that  $b$ is continuous at $\xi_0$ and $\xi_{k+1}$.

%We have these following case 
%\begin{itemize}
%    \item [\textbf{1a}]: There exists $y \in \mathbb R, y>\xi_n$ such that $b(y)< 0$. In this case we can assume $b(\xi_n)<0$.
 %   \item [\textbf{1b}]: $b(y)\geq 0,\forall y>\xi_n$. 
%    \item [\textbf{2a}]: There exists $y \in \mathbb R, y<\xi_1$ such that $b(y)> 0$. 
%    \item [\textbf{2b}]: $b(y)\leq 0,\forall y<\xi_1$. 
%\end{itemize}
%In the case $L_2<0$, due to the fact that $\lim_{x\to\infty} b(x) = - \infty$, we can assume that $b(\xi_1)> 0 > b(\xi_n)$. In this case, (1b) and (2b) occur. 

Next, we define a function $\varphi \in C^1(\mathbb R) $ as follows: 
\begin{itemize} 
\item For  $y \in [\xi_0,\xi_{k+1}]$,  
$$ \varphi(y)= b(\xi_0) + \int_{\xi_0}^ y \exp{ \left( \displaystyle\int_{\xi_0}^{x} \frac{-2b(t)}{\sigma^2(t)} \,dt \right)}
           \left[
           \displaystyle\int_{\xi_0}^{x}  \exp{ \left( \displaystyle\int_{\xi_0}^{t} \frac{2b(s)}{\sigma^2(s)} \,ds \right)} \dfrac{2R(t)}{\sigma^2(t)} \,dt 
           + K
           \right]dx, $$
where the constant $K$ is chosen such that
$$
K > 2\int_{\xi_0}^{\xi_{k+1}} \left| \exp{ \left( \int_{\xi_0}^{t} \frac{2b(s)}{\sigma^2(s)} \,ds \right)} \dfrac{2R(t)}{\sigma^2(t)}\right| \,dt + 2\exp{ \left( \int_{\xi_0}^{\xi_{k+1}} \frac{|2b(s)|}{\sigma^2(s)} \,ds \right)} + 2,
$$
and
$$
R(x) = b(\xi_0) + \dfrac{(x-\xi_0)(b(\xi_{k+1})-b(\xi_0))}{\xi_{k+1}-\xi_0},
$$
for any $x \in (\xi_0,\xi_{k+1})$.
\item For $y \in (-\infty, \xi_0)$, 
$$ \varphi(y)= \begin{cases} 
 \varphi(\xi_0) - \displaystyle\int_y^{\xi_0} \left( 1+[\varphi'(\xi_0)-1]\exp\left(\displaystyle \int_{\xi_0}^x \frac{-2b(s)}{\sigma^2(s)}ds\right) \right)dx  & \text{if } (P_b2) \text{ holds}, \\
b(\xi_0) + (y - \xi_0)\varphi'(\xi_0)
           & \text{ otherwise}. 
\end{cases}$$

\item For $y \in ( \xi_{k+1},+\infty)$, then 
$$ \varphi(y)= 
\begin{cases} 
    \varphi(\xi_{k+1}) + \displaystyle\int_{\xi_{k+1}}^y \left( 1+[\varphi'(\xi_{k+1})-1]\exp\left(\displaystyle \int_{\xi_{k+1}}^x \frac{-2b(s)}{\sigma^2(s)}ds\right)\right)dx & \text{if } (P_b1) \text{ holds},\\ 
    \varphi(\xi_{k+1}) + (x - \xi_{k+1})\varphi'(\xi_{k+1})
           & \text{otherwise}, 
\end{cases}$$
\end{itemize}

We show some useful properties of the function $\varphi$.
\begin{Lem} Under the assumption of Theorem \ref{convergence2}, the function $\varphi$ satisfies
    \begin{enumerate}
	\item[\textrm{(P1)}] $\varphi \in C^1(\mathbb R)$,
	\item[\textrm{(P2)}]  There exists a positive constant $H$ such that $1 \leq \varphi'(x) \leq H$ for any $x \in \mathbb R$,
	\item[\textrm{(P3)}]  $\varphi''$ exists and is bounded on $\mathbb R \backslash \Xi$,
	\item[\textrm{(P4)}] $\varphi'(x)b(x) + \dfrac{1}{2}\varphi''(x)\sigma^2(x) = \Psi(x)$ for any $x \in \mathbb R \backslash \Xi$, where
 \begin{equation} \label{def:Psi} 
 \Psi(x) = 
 \begin{cases}
            R(x)
           & \text{if}\quad x \in (\xi_0,\xi_{k+1}) \backslash \Xi,\\
            \varphi'(\xi_0)b(x)
           &\text{if } x \in (-\infty,\xi_0]\quad \text{ and } (P_2b) \text{ does not hold},\\
           b(x) & \text{if } x \in (-\infty,\xi_0]\quad 
           \text{ and } (P_2b) \text{ hold},\\
           \varphi'(\xi_{k+1})b(x)
           & \text{if }  x \in [\xi_{k+1},\infty) 
           \text{ and } (P_1b) \text{ does not hold},\\
           b(x) & \text{if }  x \in [\xi_{k+1},\infty)            \text{ and } (P_1b) \text{ holds}.\\
        \end{cases} 
 \end{equation} 
\end{enumerate}
\end{Lem}

\begin{proof} 
One can proves Properties (P1),(P2) and (P4) easily. To verify Property (P3), we will show that  $\sup_{x > \xi_{k+1}} |\varphi''(x)| < + \infty$. The proof for the fact  that  $\sup_{x <  \xi_0} |\varphi''(x)| < + \infty$ is similar. 

\textit{Case 1:} Suppose (A'2) holds for some $L_1 <0$ then  Property ($P_b1$) does not hold,  the result follows straightforward from the fact that  $\varphi''(x)=0$ for any $x\geq \xi_{k+1}$. 

\textit{Case 2:} Suppose that both (A6) and ($P_b1$) hold. Then for any $x>\max\{\xi'+h,\xi_{k+1}+h\}$ 
\begin{align*}
    |\varphi''(x)|&=[\varphi'(\xi_{k+1})-1]\frac{2b(x)}{\sigma^2(x)}\exp\left(\displaystyle \int_{\xi_{k+1}}^x \frac{-2b(s)}{\sigma^2(s)}ds\right) \\
    & \leq [\varphi'(\xi_{k+1})-1]\frac{2b(x)}{\sigma^2(x)}\exp\left(\displaystyle \int_{x-h}^x \frac{-2b(s)}{\sigma^2(s)}ds\right).
\end{align*}
By the mean value theorem, there exists  $\xi \in [x-h,x]$ such that
    \begin{align*}
    |\varphi''(x)| \leq& [\varphi'(\xi_{k+1})-1]\frac{2b(x)}{\sigma^2(x)}\exp\left(h \frac{-2b(\xi)}{\sigma^2(\xi)}\right)\\
    = & [\varphi'(\xi_{k+1})-1]\frac{2b(\xi)}{\sigma^2(\xi)}\exp\left(h \frac{-2b(\xi)}{\sigma^2(\xi)}\right)\\
    &  +[\varphi'(\xi_{k+1})-1]\exp\left(h \frac{-2b(\xi)}{\sigma^2(\xi)}\right)
    \frac{2b(\xi)}{\sigma^2(\xi)} \Big( \frac{\sigma^2(\xi) }{\sigma^2(x)} - 1\Big)\\
    & + [\varphi'(\xi_{k+1})-1]\exp\left(h \frac{-2b(\xi)}{\sigma^2(\xi)}\right)
    \frac{(b(x) - b(\xi))(x - \xi)}{\sigma^2(\xi) (x - \xi)}.
    \end{align*}

    Note that $\sup_{x>0} x e^{-hx} < +\infty$. Using this fact and Assumptions (A2), (A5), and (A6), we obtain the desired result. 
\end{proof}

%%%%%%%%%%

We can further see that the function $\Psi$ defined in \eqref{def:Psi} is one-sided Lipschitz continuous, i.e. there exists a constant $l_\Psi$ such that
$$(x-y)(\Psi(x)-\Psi(y))\leq l_\Psi(x-y)^2, \text{ for all } x,y \in \mathbb R.$$
We denote $\bmS :=  \cup_{i=0}^{k} (\xi_i,\xi_{i+1})^2 \cup (-\infty, \xi_0)^2 \cup (\xi_{k+1}, +\infty)^2$, and 
\begin{equation} \label{defLpsi} L_\Psi = \begin{cases} \frac{l_\Psi}{H} & \text{if } l_\Psi <0, \\ l_\Psi &\text{if } l_\Psi \geq 0.\end{cases} 
\end{equation} 

\subsubsection{Proof of Theorem \ref{convergence2}}
First, we have
$|Z_t|=|X_t-Y_t|\leq |\varphi(X_t)-\varphi(Y_t)|\leq \varepsilon+\phi_{\delta_\varepsilon}(\varphi(X_t)-\varphi(Y_t)).$
By using It\^o's formula and Property (P4), we have  
\begin{align}  \label{eq: 20}
e^{-L_\Psi t}\phi_{\delta\varepsilon}(\varphi(X_t)-\varphi(Y_t)) =J_1+J_2+J_3+J_4+J_5+J_6,
\end{align} 
where 
\begin{align*}
    J_1  &=
  \displaystyle \int_{0}^T -L_{\Psi} e^{-L_\Psi t}\phi_{\delta\varepsilon}(\varphi(X_t)-\varphi(Y_t)) \, dt, \notag \\
  J_2  & =  \displaystyle \int _0^T e^{-L_{\Psi} t}\phi'_{\delta\varepsilon}(\varphi(X_t)-\varphi(Y_t))\left(\Psi(X_t)-\Psi(Y_t)\right) \, dt,
  \notag \\
   J_3 &=   \displaystyle \int_0^T e^{-L_{\Psi }t}\phi'_{\delta\varepsilon}(\varphi(X_t)-\varphi(Y_t)) \left[ \varphi'(Y_t)(b(Y_{\underline{t}})- b(Y_t)) \right] \, dt,  \notag \\
    J_4 &=  \displaystyle \int_0^T e^{-L_{\Psi }t}\phi'_{\delta\varepsilon}(\varphi(X_t)-\varphi(Y_t)) \left[ \dfrac{1}{2} \sigma_\Delta^2(Y_{\underline{t}}) \varphi''(Y_t)
    - \dfrac{1}{2} \sigma^2(Y_t) \varphi''(Y_t) \right] \, dt,  \notag \\
    J_5 &=  \displaystyle \int_0^T \dfrac{1}{2}e^{-L_{\Psi}t}\phi''_{\delta\varepsilon}(\varphi(X_t)-\varphi(Y_t)) \left[ \varphi'(X_t)\sigma(X_t) - \varphi'(Y_t)\sigma_\Delta(Y_{\underline{t}}) \right]^2 \, dt,  \notag \\
    J_6 & = \displaystyle \int_0^T  e^{-L_{\Psi}t}\phi'_{\delta\varepsilon}(\varphi(X_t)-\varphi(Y_t))\left[ \varphi'(X_t)b(X_t) - \varphi'(Y_t)b(Y_{\underline{t}}) \right] \, dW_t. \notag
		\end{align*}
We now give a bound for each term on the left-hand side of \eqref{eq: 20}. First, by Property (YW1), we have
\begin{align*}
		\phi'_{\delta\varepsilon}(\varphi(X_t)-\varphi(Y_t))\left(\Psi(X_t)-\Psi(Y_t)\right)
  &=  \dfrac{\phi'_{\delta\varepsilon}(\varphi(X_t)-\varphi(Y_t))}{\varphi(X_t)-\varphi(Y_t)}\left(\varphi(X_t)-\varphi(Y_t)\right)\left(\Psi(X_t)-\Psi(Y_t)\right) \\
  &  =  \dfrac{\phi'_{\delta\varepsilon}(|\varphi(X_t)-\varphi(Y_t)|)}{|\varphi(X_t)-\varphi(Y_t)|}\dfrac{\varphi(X_t)-\varphi(Y_t)}{X_t - Y_t}(X_t - Y_t)\left(\Psi(X_t)-\Psi(Y_t)\right) \\
  & \leq  \dfrac{\phi'_{\delta\varepsilon}(|\varphi(X_t)-\varphi(Y_t)|)}{|\varphi(X_t)-\varphi(Y_t)|}\dfrac{\varphi(X_t)-\varphi(Y_t)}{X_t - Y_t}l_\Psi(X_t - Y_t)^2 \\
& =  l_{\Psi}\dfrac{\phi'_{\delta\varepsilon}(|\varphi(X_t)-\varphi(Y_t)|)}{|\varphi(X_t)-\varphi(Y_t)|}\left(\varphi(X_t)-\varphi(Y_t)\right)(X_t - Y_t).
\end{align*}
Using the definition of $L_\Psi$ in \eqref{defLpsi} and Property (P2), we get 
 \begin{align*}
		\phi'_{\delta\varepsilon}(\varphi(X_t)-\varphi(Y_t))\left(\Psi(X_t)-\Psi(Y_t)\right)    & \leq  L_{\Psi}\dfrac{\phi'_{\delta\varepsilon}(|\varphi(X_t)-\varphi(Y_t)|)}{|\varphi(X_t)-\varphi(Y_t)|}\left(\varphi(X_t)-\varphi(Y_t)\right)^2\\
      & =  L_{\Psi}\phi'_{\delta\varepsilon}(|\varphi(X_t)-\varphi(Y_t)|)|\varphi(X_t)-\varphi(Y_t)|.
\end{align*}
Since $-L_\Psi\phi_{\delta\varepsilon}(x)+L_\Psi|x|\phi'_{\delta\varepsilon}(|x|)\leq \max\{L_\Psi\varepsilon,0 \}$, we get 
\begin{align*}
     J_1+J_2=&\displaystyle \int_0^T \left[-L_\Psi e^{-L_\Psi t}\phi_{\delta\varepsilon}(\varphi(X_t)-\varphi(Y_t))  + e^{-L_\Psi t}\phi'_{\delta\varepsilon}(\varphi(X_t)-\varphi(Y_t))\left(\Psi(X_t)-\Psi(Y_t)\right) \right]\, dt\\
 \leq &\displaystyle \int_0^T \max\{L_\Psi  \varepsilon,0\} e^{-L_\Psi t}dt.
\end{align*}

%In the following, we denote by $C$ a constant whose value can change from one line to another. Moreover, the value of $C$ does not depend on $T$ when both $L_1$  and $\gamma$ are negative.  
For the rest of the proof, we denote by  $K_1,K_2,...$ some constants that do not depend on $\Delta$. Moreover, when $\gamma<0$, these constants do not depend on $t$ either. 

Using Properties (YW2),(P2), Assumption (A3), and Remark \ref{rm2.1}, we get
\begin{align*}
&\left|\phi'_{\delta_\varepsilon}(\varphi(X_t)-\varphi(Y_t)) \varphi'(Y_t)(b(Y_{\underline{t}})- b(Y_t))  \right|  
 \leq H\left| b(Y_{\underline{t}})- b(Y_t)  \right| \\
& = H\left| b(Y_{\underline{t}})- b(Y_t)  \right|1_{ \bmS}(Y_t,Y_{\underline{t}}) + H\left| b(Y_{\underline{t}})- b(Y_t)  \right|1_{\mathbb{R}^2\backslash \bmS}(Y_t,Y_{\underline{t}})\\
& \leq K_1 \left( 1_{ \bmS}(Y_t,Y_{\underline{t}}) +  (1+|Y_t|^l+|Y_{\underline{t}}|^l )|Y_t-Y_{\underline{t}}|\right).
\end{align*}
%with the note that $S= \Big(\cup_ {i=0}^{k}(\xi_i,\xi_{i+1})^2\Big)^c$
Then, Proposition \ref{dinh ly 2}, Lemma \ref{hq2} and Lemma \ref{tichphanxacsuatchuyen2} give
\begin{equation*}
   \mathbb E[J_3]\leq K_1 \displaystyle \int_0^T  e^{-L_\Psi t}\mathbb E\left(1_{ \bmS}(Y_t,Y_{\underline{t}}) + (1+|Y_t|^l+|Y_{\underline{t}}|^l )|Y_t-Y_{\underline{t}}|\right) dt \leq
    K_2 \displaystyle \int_0^T e^{-L_\Psi t}\Delta^{\frac{1}{4}+\frac{\alpha}{2}} \, dt.
\end{equation*}
 For $J_4$, we first note that  
\begin{align*}
&\left| \phi'_{\delta_\varepsilon}(\varphi(X_t)-\varphi(Y_t)) \left[  \sigma_\Delta^2(Y_{\underline{t}}) \varphi''(Y_t)
    -  \sigma^2(Y_t) \varphi''(Y_t) \right] \right| \\
    &\leq
    K_3 \left[(1+|Y_t|^{2m+\alpha+1/2}+|Y_{\underline{t}}|^{2m+\alpha+1/2})|Y_t-Y_{\underline{t}}|^{1/4+\alpha/2} + (1+|Y_{\underline{t}}|^{2m+2\alpha+1})\Delta^{1/2} \right];
\end{align*} 
hence, by Proposition \ref{dinh ly 2} and Lemma \ref{hq2}
\begin{align*}
   \mathbb E[J_4]\leq &\displaystyle \int_0^T \mathbb E \left[ \Bigg| e^{-L_\Psi t}\phi'_{\delta_\varepsilon}(\varphi(X_t)-\varphi(Y_t)) \left( \dfrac{1}{2} \sigma_\Delta^2(Y_{\underline{t}}) \varphi''(Y_t)
    - \dfrac{1}{2} \sigma^2(Y_t) \varphi''(Y_t) \right)\Bigg| \right] \, dt\\
    \leq & K_4 \displaystyle \int_0^T e^{-L_\Psi t} \mathbb E \left[(1+|Y_t|^{2m+\alpha+1/2}+|Y_{\underline{t}}|^{2m+\alpha+1/2})|Y_t-Y_{\underline{t}}|^{1/2+\alpha} + (1+|Y_{\underline{t}}|^{2m+2\alpha+1})\Delta^{1/2} \right]dt\\
    \leq &K_5 \displaystyle \int_0^T  e^{-L_\Psi t}(\Delta^{1/4+\alpha/2}+\Delta^{1/2}) \, dt.
\end{align*}
For $J_5$, using the estimate $(a+b+c+d)^2 \leq 4(a^2 + b^2 + c^2 + d^2)$, we get  
\begin{align*}
&\displaystyle\int_0^T e^{-L_\Psi t}\phi''_{\delta_\varepsilon}(\varphi(X_t)-\varphi(Y_t)) \left[ \varphi'(X_t)\sigma(X_t) - \varphi'(Y_t)\sigma_\Delta(Y_{\underline{t}}) \right]^2dt\\ &
\leq 4\displaystyle\int_0^T e^{-L_\Psi t} \phi''_{\delta_\varepsilon}(\varphi(X_t)-\varphi(Y_t)) \left[ \varphi'(X_t) \right]^2 \left[\sigma(X_t) - \sigma(Y_t) \right]^2dt\\
&
+ 4 \displaystyle\int_0^T e^{-L_\Psi t}\phi''_{\delta_\varepsilon}(\varphi(X_t)-\varphi(Y_t)) \left[ \varphi'(Y_t) \right]^2 \left[\sigma(Y_t) - \sigma(Y_{\underline{t}}) \right]^2dt\\
&
+ 4 \displaystyle\int_0^T e^{-L_\Psi t}\phi''_{\delta_\varepsilon}(\varphi(X_t)-\varphi(Y_t)) \left[ \varphi'(Y_t) \right]^2 \left[\sigma_\Delta(Y_{\underline{t}} - \sigma(Y_{\underline{t}}) \right]^2dt\\
&
+ 4\displaystyle\int_0^T e^{-L_\Psi t} \phi''_{\delta_\varepsilon}(\varphi(X_t)-\varphi(Y_t)) \sigma^2(Y_t)  \left[\varphi'(X_t) - \varphi'(Y_t) \right]^2dt\\
% & \leq D_1.\Delta^\alpha.(1+|X_t|^m+|Y_t|^m) + D_2.\Delta^\frac{1}{2}
% .\sigma^2(Y_t)+D_3.\Delta^\frac{1}{2}. 
&:= 4(J_{5,1}+J_{5,2}+J_{5,3}+J_{5,4}).
\end{align*}
Using Properties (YW5), (P2), Assumption (A4), Proposition \ref{dinh ly 2} and  Proposition \ref{nghiem dung 1}, we get 
\begin{align*}
    \mathbb E[J_{5,1}] \leq &H^2 \displaystyle\int_0^T e^{-L_\Psi t}\mathbb E\left[\phi''_{\delta\varepsilon}(\varphi(X_t)-\varphi(Y_t))|\varphi(X_t)-\varphi(Y_t)|\frac{\big[\sigma(X_t)-\sigma(Y_t)\big]^2}{|\varphi(X_t)-\varphi(Y_t)|}\right]dt\\
    &\leq K_6 \displaystyle\int_0^T e^{-L_\Psi t} \mathbb E\left[\frac{2}{ \log \delta}1_{\{|\varphi(X_t)-\varphi(Y_t)|\leq \varepsilon\}}\frac{|X_t-Y_t|^{1+2\alpha}(1+|X_t|^{2m}+|Y_t|^{2m})}{|\varphi(X_t)-\varphi(Y_t)|}\right]dt\\
    &\leq K_7 \displaystyle\int_0^T e^{-L_\Psi t}\frac{\varepsilon^{2\alpha}}{\log \delta}\mathbb E\left[(1+|X_t|^{2m}+|Y_t|^{2m})\right]dt\\
    &\leq K_8 \displaystyle\int_0^T e^{-L_\Psi t}\frac{\varepsilon^{2\alpha}}{\log\delta} \,dt,
\end{align*}
and 
\begin{align*}
   \mathbb E[ J_{5,4}] &=\displaystyle\int_0^T e^{-L_\Psi t} \mathbb E\left[\phi''_{\delta\varepsilon}(\varphi(X_t)-\varphi(Y_t))|\varphi(X_t)-\varphi(Y_t)|\sigma^2(Y_t)\frac{|\varphi'(X_t)-\varphi'(Y_t)|^2}{|\varphi(X_t)-\varphi(Y_t)|^2}|\varphi(X_t)-\varphi(Y_t)|  \right] \,dt\\
    &\leq K_9 \displaystyle\int_0^T e^{-L_\Psi t}\mathbb E\left[ \frac{2}{\log \delta}1_{\{|\varphi(X_t)-\varphi(Y_t)|\leq \varepsilon\}}|\varphi(X_t)-\varphi(Y_t)|\sigma^2(Y_t) \right] \,dt\\
    &\leq K_{10} \displaystyle\int_0^T e^{-L_\Psi t} \frac{2}{\log \delta}\varepsilon\mathbb E\left[\sigma^2(Y_t)\right] \,dt\\
    & \leq K_{11} \displaystyle\int_0^T e^{-L_\Psi t} \frac{2\varepsilon}{\log \delta} \,dt.
\end{align*}
Using Properties (YW5), (P2), Assumption (A4), Proposition \ref{dinh ly 2}, and Lemma \ref{hq2}, we get 
\begin{align*}
    \mathbb E[J_{5,2}]\leq K_{12} \displaystyle\int_0^T e^{-L_\Psi t} \mathbb E\left[\big(\sigma(Y_t)-\sigma(Y_{\underline t})\big)^2\frac{2\delta}{\varepsilon\log\delta} \,\right] dt \leq K_{13} \displaystyle\int_0^T e^{-L_\Psi t}\Delta^{\alpha+1/2}\frac{2\delta}{\varepsilon\log\delta} \, dt.
\end{align*}
Using similar estimates as in \eqref{sigmadelta}, we get 
\begin{align*}
    \mathbb E[J_{5,3}]\leq K_{14} \displaystyle\int_0^T e^{-L_\Psi t}\mathbb E\left[\frac{2\delta}{\varepsilon\log\delta}\left(\sigma_\Delta(Y_{\underline{t}}) - \sigma(Y_{\underline{t}}) \right)^2\right] \, dt\leq K_{15} \displaystyle\int_0^T e^{-L_\Psi t} \frac{2\delta}{\varepsilon\log\delta}\Delta \,dt.
\end{align*}
Finally,
$\mathbb E[J_6]=0.$
To sum up, 
\begin{align*}\mathbb E[e^{-L_\Psi T}|Z_T|]& \leq \mathbb E[J_1+J_2+J_3+J_4+J_5+J_6]+ e^{-L_\Psi T}\varepsilon \\
&\leq K_{16}\displaystyle\int_0^T e^{-L_{\Psi}t}\left[\varepsilon+\Delta^{1/2} +\Delta^{\frac 14 + \frac{\alpha}{2}}+ \dfrac{ \varepsilon^{2\alpha}}{\log \delta}+\dfrac{\delta \Delta^{1/2+\alpha}}{\varepsilon \log \delta} + \dfrac{\delta \Delta }{\varepsilon \log \delta}+\frac{2\varepsilon}{\log \delta} \right] \,dt.
\end{align*}
If $\alpha \in (0,\frac{1}{2}]$, then by choosing $\varepsilon=\Delta^{1/2}$ and $\delta=2$, we obtain
\begin{equation} \label{thang1}
\mathbb E[|Z_T|] \leq C \frac{e^{L_\Psi T} -1}{L_\Psi} \Delta^\alpha.
\end{equation} 
%			\begin{align*}
%			\sup_{t\in[0,T]} \bE \left[|Z_t|\right] \le C \Delta^{\alpha}. 
%			\end{align*}
			 If $\alpha=0$, then by choosing $\varepsilon=\Delta^{1/4}$ and $\delta=\Delta^{-1/4}$, we obtain
			\begin{equation} \label{thang2}
			\sup_{t\in [0,T]} \bE \left[|Z_t|\right] \le C \frac{e^{L_\Psi T} -1}{L_\Psi} \frac{1}{\log \frac{1}{\Delta}}. 
			\end{equation}
Recall that the constants $C$ in \eqref{thang1} and \eqref{thang2} do not depend on $T$ when $L_1, \gamma$ and $L_\Psi$ are negative. This implies the desired result. 
% Therefore,
% \begin{equation}
%     \dfrac{1}{2}e^{-L_Tt}\phi''_{\delta_\varepsilon}(\varphi(X_t)-\varphi(Y_t)) \left[ \varphi'(X_t)\sigma(X_t) - \varphi'(Y_t)\sigma_\Delta(Y_{\underline{t}}) \right]^2 \, dt
%     \leq e^{-L_T.t}\Delta^\alpha \, dt.
% \end{equation}
% (to be continued...?)

%If $L_1<0$ and $\gamma<0$, the constant $L_\Psi<0$, which implies the constant $C$ does not depend on $T$.

%\subsection{Cost analysis}

    \subsection{Proof of Theorem \ref{cost analysis}}
    
    \begin{proof}
    We first note that 
    $$N_{T}=1+\sum_{k=1}^{\infty}1_{\{t_k<T\}}=1+\sum_{k=1}^{\infty}1_{\{t_k<T\}}\displaystyle \int_{t_{k-1}}^{t_{k}}\frac{1}{h_{\Delta}(\overline{Y}_s)}ds=1+\displaystyle \int _{0}^T\frac{1}{h_{\Delta}(\overline{Y}_s)}ds.$$
We write $\mathbb E\Big[\displaystyle \int _{0}^T\frac{1}{h_{\Delta}(\overline{Y}_s)}ds\Big]= I_1 + I_2 + I_3$, where 
    $ I_1 =\mathbb E\Big[\displaystyle \int _{0}^T\frac{1}{h_{\Delta}(\overline{Y}_s)}1_{(\Xi^{\varepsilon_1})^c}(Y_{\underline{s}})ds\Big]$,\\ 
    $I_2 = \mathbb E\Big[\displaystyle \int _{0}^T\frac{1}{h_{\Delta}(\overline{Y}_s)}1_{\Xi^{\varepsilon_1}\setminus\Xi^{\varepsilon_2}}(Y_{\underline{s}})ds\Big],$ 
    $I_3  = \mathbb E\Big[\displaystyle \int _{0}^T\frac{1}{h_{\Delta}(\overline{Y}_s)}1_{\Xi^{\varepsilon_2}}(Y_{\underline{s}})ds\Big].$
 
 For the rest of the proof, we denote by  $K_1,K_2,...$ some constants that do not depend on $\Delta$. Moreover, when $\gamma<0$, these constants do not depend on $t$ either. 
    It follows from \eqref{chooseh} and  Proposition \ref{dinh ly 2}  that 
    \begin{align*}
        I_1&\leq\mathbb E\Big[\displaystyle \int _{0}^T\frac{\big[1+|b(Y_{\underline{s}})|+|\sigma(Y_{\underline{s}})|+|Y_{\underline{s}}|^l\big]^2}{\Delta}ds\bigg]       \leq K_1 T\Delta^{-1}. \label{I1}
    \end{align*}
     Thanks to Assumptions (A3) and (A4),   
     $\sup_{y\in \Xi^{\varepsilon_2}}\Big[1+|b(y)|+|\sigma(y)|+|y|^l\Big]^2 < +\infty.$  It follows from \eqref{chooseh} that 
   %  Applying Lemma \ref{localtime} with $f(x)=1$ and $\varepsilon=\varepsilon_2$, we have 
   \begin{equation}
    \begin{aligned}
        I_3    &=\mathbb E\Big[\displaystyle \int _{0}^T\frac{\big[1+b(Y_{\underline{s}})+\sigma(Y_{\underline{s}})+|Y_{\underline{s}}|^l\big]^2}{\Delta^2\log^4(1/\Delta)}1_{\Xi^{\varepsilon_2}}(Y_{\underline{s}})ds\Big]
    \leq \frac{K_2}{\Delta^2\log^4(1/\Delta)} \displaystyle\int_0^T\mathbb P\Big(Y_{\underline{s}}\in \Xi^{\varepsilon_2}\Big)ds. 
    \end{aligned} \label{est:I3}
    \end{equation}
    Similarly, we have 
    \begin{align*}
        I_2&=\mathbb E\Big[\displaystyle \int _{0}^T\frac{\log^4(1/\Delta)\big[1+|b(Y_{\underline{s}})|+|\sigma(Y_{\underline{s}})|+|Y_{\underline{s}}|^l\big]^2}{\big[d(Y_{\underline{s}},\Xi)\big]^2}1_{\Xi^{\varepsilon_1}\setminus \Xi^{\varepsilon_2}}(Y_{\underline{s}})ds\bigg]\\
        &\leq K_3 \log^4(1/\Delta)\mathbb E\Big[\displaystyle \int _{0}^T\frac{1}{\big[d(Y_{\underline{s}},\Xi)\big]^2}1_{\Xi^{\varepsilon_1}\setminus \Xi^{\varepsilon_2}}(Y_{\underline{s}})ds\bigg] =  K_3 \log^4(1/\Delta)(I_{21} + I_{22}),
        \end{align*}
        where 
        \begin{align*}
      I_{2,1} & = \mathbb E\Big[\displaystyle \int _{0}^T\frac{1}{\big[d(Y_{\underline{s}},\Xi)\big]^2}1_{\Xi^{\varepsilon_1}\setminus \Xi^{\varepsilon_2}}(Y_{\underline{s}})1_{\{|Y_s-Y_{\underline{s}}|\geq \frac{1}{2}d(Y_{\underline{s}},\Xi)\}}ds\bigg],\\    I_{2,2}& = \mathbb E\Big[\displaystyle \int _{0}^T\frac{1}{\big[d(Y_{\underline{s}},\Xi)\big]^2}1_{\Xi^{\varepsilon_1}\setminus \Xi^{\varepsilon_2}}(Y_{\underline{s}})1_{\{|Y_s-Y_{\underline{s}}|< \frac{1}{2}d(Y_{\underline{s}},\Xi)\}}ds\bigg].
    \end{align*}
  
   By applying Lemma \ref{xac suat chuyen}.ii, for all $\Delta<\Delta_0$ we obtain
    \begin{align*}
        I_{2,1}%&=\mathbb E\Big[\displaystyle \int _{0}^T\frac{1}{\big[d(Y_{\underline{s}},\Xi)\big]^2}1_{\Xi^{\varepsilon_1}\setminus \Xi^{\varepsilon_2}}(Y_{\underline{s}})1_{\{|Y_s-Y_{\underline{s}}|\geq \frac{1}{2}d(Y_{\underline{s}},\Xi)\}}ds\bigg]\\
        &\leq \frac{1}{\varepsilon_2^2}\displaystyle\int_{0}^T \mathbb P\Big(|Y_s-Y_{\underline{s}}|\geq \frac 12 d(Y_{\underline{s}},\Xi),  Y_{\underline{s}} \in \Xi^{\varepsilon_1}\setminus \Xi^{\varepsilon_2}\Big)ds \leq K_4 \frac{T}{\Delta^2\log^8(1/\Delta)}\Delta^{\frac 12\log(1/\Delta)}.
     \end{align*}    
     
       Note that $|Y_s-Y_{\underline{s}}|< \frac 12 d(Y_{\underline{s}},\Xi)$ implies $\frac 12 d(Y_{\underline{s}},\Xi)\leq d(Y_s,\Xi)\leq \frac 32 d(Y_{\underline{s}},\Xi)$. Hence, if $\varepsilon_2\leq d(Y_{\underline{s}},\Xi)\leq \varepsilon_1$, then $\frac 12 \varepsilon_2\leq d(Y_s,\Xi)\leq \frac 32 \varepsilon_1.$ Hence
    \begin{align*}
        I_{2,2}        &\leq  \frac 94\mathbb E\Big[\displaystyle \int _{0}^T\frac{1}{\big[d(Y_{s},\Xi)\big]^2}1_{ \Xi^{\frac 32 \varepsilon_1}\setminus \Xi^{\frac 12\varepsilon_2}}(Y_{s})ds\Big].
    \end{align*}
    
Now we will consider two cases. 

\vskip 0.2cm
\noindent 
\textit{Case 1}: $[\frac{p_0}{2}] >  \frac{3}{2\alpha + 1} (1 + 2\alpha + 2m)$.

%In \cite{Larisa}, they estimate $I_{2,1}$ under the condition that $\sigma$ is Lipschitz continuous. They divide the domain $\frac 12 d(Y_{\underline{s}},\Xi)\leq d(Y_s,\Xi)\leq \frac 32$ into 2 subdomains. We now improve it by dividing the domain into more than 2 subdomains.\\

We can choose $n$ such that $n > \frac{3}{2\alpha + 1} $ and $[\frac{p_0}{2}] >  \frac{3}{2\alpha + 1} (1 + 2\alpha + 2m)n$, and then choose $p \in (\frac{4}{2\alpha +1}, \frac{4n}{3})$. Let  $q=\frac{4n-3}{4n}$.  Define the sequence $a_i$ as follow: $a_0=\frac 12$, and 
$a_{i+1}=\frac12(a_i(q+\frac{\alpha}{2}+\frac 14)+1)$ for $i\geq0$. Let $\varepsilon_{a_i}:=\Delta^{a_i}\log^{4a_i}(1/\Delta)$. 
Since $q+\frac{\alpha}{2}+\frac 14>1$, we have $\lim\limits_{i\rightarrow \infty}a_i>1$. Thus, there exists a positive integer $h$ such that  $a_h\leq 1<a_{h+1}$. 
 Then we write 
 \begin{align*}
       I_{2,2}
        &\leq \frac 94\mathbb E\Big[\displaystyle \int _{0}^T\frac{1}{\big[d(Y_{s},\Xi)\big]^2}1_{ \Xi^{\frac 32 \varepsilon_1}\setminus \Xi^{ \varepsilon_{a_1}}}(Y_s)ds\Big]+\frac 94\mathbb E\Big[\displaystyle \int _{0}^T\frac{1}{\big[d(Y_{s},\Xi)\big]^2}1_{ \Xi^{ \varepsilon_{a_h}}\setminus \Xi^{\frac 12\varepsilon_2}}(Y_s)ds\Big]\\
        &+\frac 94 \sum\limits_{i=1}^{h-1}\mathbb E\Big[\displaystyle \int _{0}^T\frac{1}{\big[d(Y_{s},\Xi)\big]^2}1_{ \Xi^{ \varepsilon_{a_{i}}}\setminus \Xi^{\varepsilon_{a_{i+1}}}}(Y_s)ds\Big]\\
        &\leq\frac 94 \sum\limits_{i=0}^{h-1}\mathbb E\Big[\displaystyle \int _{0}^T\frac{1}{\big[d(Y_{s},\Xi)\big]^2}1_{ \Xi^{\frac 32 \varepsilon_{a_{i}}}\setminus \Xi^{\frac 12\varepsilon_{a_{i+1}}}}(Y_s)ds\Big]
        +\frac 94 \mathbb E\Big[\displaystyle \int _{0}^T\frac{1}{\big[d(Y_{s},\Xi)\big]^2}1_{ \Xi^{\frac 32 \varepsilon_{a_{h}}}\setminus \Xi^{\frac 12\varepsilon_2}}(Y_s)ds\Big].
\end{align*}
Applying  Lemma \ref{localtime} with $f(x)=\frac{1}{(\text{max}\{\frac 12 \varepsilon_{a_{i+1}}, x\})^2},\varepsilon=\frac 32\varepsilon_{a_{i}}$, we have that for each $i=0,\ldots,h-1,$
\begin{align*}
   & \mathbb E\Big[\displaystyle \int _{0}^T\frac{1}{\big[d(Y_{s},\Xi)\big]^2}1_{ \Xi^{\frac 32 \varepsilon_{a_{i}}}\setminus \Xi^{\frac 12\varepsilon_{a_{i+1}}}}(Y_s)ds\Big]\leq \mathbb E\Big[\displaystyle \int _{0}^T\frac{1}{\max\big\{\frac 12 \varepsilon_{a_{i+1}},d(Y_{s},\Xi)\big\}^2}1_{ \Xi^{\frac 32 \varepsilon_{a_{i}}}\setminus \Xi^{\frac 12\varepsilon_{a_{i+1}}}}(Y_s)ds\Big]\\
    &\leq  K_5 T\displaystyle \int_{0}^{\frac 32 \varepsilon_{a_{i}}}\frac{1}{(\text{max}\{ \frac 12\varepsilon_{a_{i+1}}, x\})^2}dx+K_5  T \frac{4}{\varepsilon_{a_{i+1}}^2}\Big(\big(\frac 32 \varepsilon_{a_{i}}\big)^{q+\frac{\alpha}{2}+\frac 14}+\Delta+\Delta^{(\frac 14+\frac{\alpha}{2})p}\Big) \\
    & \leq K_6 T(\Delta^{-1}+\Delta^{(\frac 14 +\frac{\alpha}{2})p-2})\frac{1}{\log^4(1/\Delta)}\\
    & \leq 2K_6 T \Delta^{-1}\frac{1}{\log^4(1/\Delta)},
\end{align*}
where the last estimate follows from the fact that $(\frac 14+\frac{\alpha}{2})p  > 1$.\\
Similarly, applying  Lemma \ref{localtime} with $f(x)=\frac{1}{(\text{max}\{\frac 12 \varepsilon_{2}, x\})^2},\varepsilon=\frac 32\varepsilon_{a_{h}}$, we have 
\begin{align*}
   & \mathbb E\Big[\displaystyle \int _{0}^T\frac{1}{\big[d(Y_{s},\Xi)\big]^2}1_{ \Xi^{\frac 32 \varepsilon_{a_{h}}}\setminus \Xi^{\frac 12\varepsilon_2}}(Y_s)ds\Big]\leq \mathbb E\Big[\displaystyle \int _{0}^T\frac{1}{\max\big\{\frac 12 \varepsilon_2,d(Y_{s},\Xi)\big\}^2}1_{ \Xi^{\frac 32 \varepsilon_{a_{h}}}\setminus \Xi^{\frac 12\varepsilon_{2}}}(Y_s)ds\Big]\\
    &\leq  K_7 T\displaystyle \int_{0}^{\frac 32 \varepsilon_{a_{h}}}\frac{1}{(\text{max}\{ \frac 12\varepsilon_{2}, x\})^2}dx+K_7 T \frac{4}{\varepsilon_{2}^2}\Big(\big(\frac 32 \varepsilon_{a_h}\big)^{q+\frac{\alpha}{2}+\frac 14}+\Delta+\Delta^{(\frac 14+\frac{\alpha}{2})p}\Big) \\
    & \leq {K_8 T(\Delta^{-1}+\Delta^{(\frac 14 +\frac{\alpha}{2})p-2})}\frac{1}{\log^4(1/\Delta)}\\
     & \leq K_9 T \Delta^{-1}\frac{1}{\log^4(1/\Delta)}.
\end{align*}
   To estimate $I_3$, applying Lemma \ref{localtime} with $f(x)=1$ and $\varepsilon=\varepsilon_2$, it follows from \eqref{est:I3} that  
    \begin{align*}
        I_3       &\leq \frac{K_{10} T}{\Delta^2\log^4(1/\Delta)} \Big(\varepsilon_2+(\varepsilon_2+\Delta+\Delta^{(\frac 14+\frac{\alpha}{2})p})\Big) \leq  K_{11} \frac{T}{\Delta \log^4(1/\Delta) }.
    \end{align*}

\vskip 0.2cm 
\noindent     
    \textit{Case 2}: $[\frac{p_0}{2}] \leq   \frac{3}{2\alpha + 1} (1 + 2\alpha + 2m)$.
    
    Applying  Lemma \ref{localtime} with $f(x)=\frac{1}{(\text{max}\{\frac 32 \varepsilon_2, x\})^2},\varepsilon=\frac 32 \varepsilon_1$, for all $\Delta<\Delta_0$, we have
    \begin{align*}
    I_{2,2}&\leq \frac 94\mathbb E\Big[\displaystyle \int _{0}^T\frac{1}{\big[d(Y_{s},\Xi)\big]^2}1_{ \Xi^{\frac 32 \varepsilon_1}}(Y_{\underline{s}})ds\Big]
    \leq K_{12}T\displaystyle \int_{0}^{\frac 32 \varepsilon_1}\frac{1}{(\text{max}\{\frac 32 \varepsilon_2, x\})^2}dx+K_{12}T\Delta^{\frac 14 + \frac{\alpha}{2}}\frac{1}{\varepsilon_2^2}\\
    & \leq K_{13}T \frac{1}{\log^4(1/\Delta)} \Delta^{\alpha/2 - 7/4}.
    \end{align*}
    
%  \textbf{The case $\frac{\alpha}{2}+\frac 14-\frac{3}{4n}<\frac{1}{12}$}
 It follows from Lemma \ref{localtime} with $f(x) = 1$ and $\varepsilon = \varepsilon_2$, and \eqref{est:I3} that 
    \begin{align*}
        I_3    &\leq  K_{14} \frac{T}{\Delta\log^4(1/\Delta)}(\Delta^{\frac 14 + \frac{\alpha}{2}}+\Delta^2 \log^4(1/\Delta)) \leq K_{15} T\Delta^{\alpha/2-7/4}.
    \end{align*}
     When $\gamma<0$,  it follows from the uniform boundedness of the moment of $Y$ that the constants $(K_i)_{1\leq i \leq 15}$ does not depend on $T$.
    We conclude the proof.
    \end{proof}

\subsection{Proofs of Theorem \ref{existence1} and Theorem \ref{existence2}}
The following lemma is needed for the proofs of Theorem \ref{existence1} and Theorem \ref{existence2}.
\begin{Lem}\label{uniquess}
    Suppose that Assumptions (A1), (A2) and (A4) hold for some $m \geq 0$. Moreover,  there exists a constant $C>0$ such that for any solution $X$ of the equation \eqref{eqn1.1}, we have
    $$\sup\limits_{t\in[0,T]} \mathbb E[|X_t|^{2m\vee 1}]\leq C.$$
%    where the constant $m$ is defined in Assumption (A4). 
Then,  the equation \eqref{eqn1.1} has at most one  strong solution.
\end{Lem}
\begin{proof}
%It is clearly that for any solution of \eqref{eqn1.1}, it is integrable.
Assume that $(X')$ is another solution of the equation \eqref{eqn1.1}, we will show that $\bE[|X_t - X'_t|] = 0$ for all $t \in [0,T]$, which implies the uniqueness of the solution. \\
%Since $X_t,X'_t$ are solutions of \eqref{eqn1.1}, 
%$$X_{t} -X'_{t}=\int_0^t \left[b(X_{s})-b(X'_{s})\right]ds+\int_0^t \left[\sigma(X_{s})-\sigma(X'_{s})\right]dW_s.$$
By applying It\^o's formula for $\phi_{\delta\varepsilon}(X_{t} -X'_{t})$ and using Property (YW3), we have
\begin{align}
\left|X_{t} -X'_{t}\right|&\leq \varepsilon+\int_0^t \phi'_{\delta\varepsilon}\left(X_s -X'_s\right)\left[b(X_s)-b(X'_s)\right]ds \notag \\
&+\frac{1}{2}\int_0^t \phi''_{\delta\varepsilon}\left(X_s -X'_s\right)\left[\sigma(X_s)-\sigma(X'_s)\right]^2ds \notag \\
&+\int_0^{t} \phi'_{\delta\varepsilon}\left(X_s -X'_s\right)\left[\sigma(X_{s})-\sigma(X'_{s})\right]dW_s. \label{tag14}
\end{align}
Using Assumption (A2) and Properties (YW1), (YW2), we get 
\begin{align}\label{estimateoneside}
    \displaystyle \int _0^t \phi'_{\delta\varepsilon}(X_s-X'_s)[b(X_s)-b(X'_s)]ds=&\displaystyle \int _0^t \frac{\phi'_{\delta\varepsilon}(X_s-X'_s)}{X_s-X'_s}(X_s-X'_s)[b(X_s)-b(X'_s)]ds \\ \notag
    &\leq \displaystyle \int _0^t \frac{\phi'_{\delta\varepsilon}(|X_s-X'_s|)}{|X_s-X'_s|}L_1(X_s-X'_s)^2ds\\ \notag
    &\leq \displaystyle \int _0^t |L_1||X_s-X'_s|ds. \notag
\end{align}

Using Assumption (A4), Property (YW5), and Proposition \ref{nghiem dung 1}, we have
\begin{align*}
\mathbb E \bigg[\int_0^t &\phi''_{\delta\varepsilon}\left(X_s -X'_s\right)\left[\sigma(X_s)-\sigma(X'_s)\right]^2ds\bigg]
%&\leq \int_0^t \frac{2}{\log\delta \left|X_s -X'_s\right|}1_{\left[\frac{\varepsilon}{\delta};\varepsilon\right]}\left(|X_s-X'_s|\right)\left[\sigma(X_s)-\sigma(X'_s)\right]^21_{[s\leq \theta_R]}ds \notag \\
%&\leq \int_0^t \frac{2L_R}{\log\delta \left|X_s -X'_s\right|}1_{\left[\frac{\varepsilon}{\delta};\varepsilon\right]}\left(|X_s-X'_s|\right)\left|X_s-X'_s\right|^{1+2\alpha}1_{[s\leq \theta_R]}ds \notag \\
\leq  \frac{2\varepsilon^{2\alpha}}{\log \delta}\mathbb E\left[\displaystyle \int _0^t 3L_3+3L_3|X_s|^{2m}+3L_3|X'_s|^{2m}ds\right] \leq \frac{K_1\varepsilon^{2\alpha}}{\log \delta},
\end{align*}
where $K_1$  depend  on neither $\delta$ nor $\varepsilon$. Together with  \eqref{tag14} and \eqref{estimateoneside}, this implies 
$$\bE[\left|X_{t} -X'_{t}\right|]\leq \varepsilon+|L_1|\int_0^t \bE[ \left|X_{s} -X'_{s}\right|]ds+\frac{K_1\varepsilon^{2\alpha}}{\log \delta}.$$
By letting $\varepsilon\to 0,\delta \to \infty $, we get 
$$\bE[\left|X_{t} -X'_{t}\right|] \leq |L_1|\int_0^t \bE[\left|X_{s} -X'_{s}\right|]ds.$$
By Gronwall's inequality, we induce that
$\bE\left[\left|X_{t} -X'_{t}\right|\right]=0.$
 The proof is complete.
\end{proof}
\subsubsection*{Proof of Theorem \ref{existence1}}
%It has been shown in [Veretennikov 1981] that if b is a bounded mearsuable function and (A4),(A5') hold then \eqref{eqn1.1} has unique strong solution.
 We will use the localization technique as in \cite{NL17}. For each $N>1$, set 
\begin{equation*}
    b_N(x)=
    \begin{cases}
        b(x)  &\text{ if } |x|\leq N,\\
        b\Big(\frac{Nx}{|x|}\Big)(N+1-|x|\Big)  &\text{ if }  N<|x|<N+1,\\
        0  &\text{ if }  |x|\geq N+1,
    \end{cases}
\end{equation*}
and
\begin{equation*}
    \sigma_N(x)=
    \begin{cases}
        \sigma(x) & \text{ if }  |x|\leq N,\\
        \Big(\sigma\Big(\frac{Nx}{|x|}\Big)-1\Big)(N+1-|x|)+1  & \text{ if }  N<|x|<N+1,\\
        1  &\text{ if } |x|\geq N+1.
    \end{cases}
\end{equation*}
It can verify that $b_N$ is bounded, and $\sigma_N$ is H\"older continuous and uniformly elliptic. 
Then the equation
$$X_N(t) = x_0 + \int_0^t b_N(X_s)ds+ \int_0^t \sigma_N(X_s)dW_s$$
has a unique strong solution.
 Moreover, we can verify that $xb_N(x)+\frac{p_0-1}{2}|\sigma_N(x)^2|\leq \gamma'|x|^2+\eta'$ for some constants $\gamma',\eta'$ depending only on  $\gamma,\eta$.
Hence, it follows from Lemma 3.1 in [LT19] that if $p_0\geq (l+4)\vee (2m+2\alpha+4)$ and $2\leq p\leq (p_0-(l\vee (2m+2\alpha)))/2$, then
\begin{equation*}
\mathbb E\Big[\sup\limits_{t\in [0,T]}|X_N(t)|^p\Big]\leq C(x,T,l,p,\gamma,\eta,m).
\end{equation*}
Using the argument in the proof of Theorem 3.1.i in \cite{NL17}, we can show that when $N\to \infty$, $X_N$ will converge in probability to a process $X$ which satisfies the equation \eqref{eqn1.1}. \\ 
It remains to prove the uniqueness of the solution. Let $\varphi$ be the function defined in Theorem \ref{convergence2}. Note that from the properties (P1)--(P3), $\varphi^{-1}$ exists and has bounded and Lipschitz continuous derivative. Let   $\overline b := \Psi\circ \varphi^{-1}$ and $\overline{\sigma}:=(\varphi'.\sigma)\circ\varphi^{-1}$. 
It can be verified that $\overline{b}$ and $\overline{\sigma}$ satisfy the following properties: for any $x, y \in \mathbb{R}$,
$$(x-y)(\overline{b}(x)-\overline{b}(y))\leq C(x-y)^2, $$ and
$$|\overline{\sigma}(x)-\overline{\sigma}(y)| \leq C|x-y|^{\alpha+1/2}(1+|x|^{m+1}+|y|^{m+1})$$
for some positive constant $C$.
Let $U_t:=\varphi(X_t)$. Using It\^o's formula and Property (P4), we obtain
\begin{equation} \label{pt:Z}
dU_t=\overline b(U_t)dt+\overline{\sigma}(U_t)dW_t.
\end{equation} 
It follows from Lemma \ref{nghiem dung 1} and Property (P2) that 
$$\mathbb E[|U_t|^{2m+2}]\leq C_1+C_1\mathbb E[|X_t|^{2m+2}]< +\infty. $$
Then it follows from Lemma \ref{uniquess} that there exists at most one solution to equation \eqref{pt:Z}. Since  $\varphi$ is strictly increasing, we obtain the uniqueness for solution of equation \eqref{eqn1.1}. This concludes the proof of Theorem \ref{existence1}.

\subsubsection*{Proof of Theorem \ref{existence2}}
For the proof of Theorem \ref{existence2}, we need the following result.
\begin{Lem}\label{existence2.1}
  Suppose that Assumptions (A1)--(A5) hold for  $m=l=\alpha=0$, and $b$ is bounded by $\|b\|_{\infty}<+\infty$.  Then the equation \eqref{eqn1.1} has a unique strong solution.
\end{Lem}
\begin{proof}
   For each $N> \frac{1}{\varepsilon_0}$, we define
    \begin{equation}
        b_N(x)=\begin{cases}
           & b(x) \quad \text{if} \quad d(x,\Xi)>\frac 1N,\\
           & b(\xi_i-\frac 1N)+\frac{ \big(x-\xi_i+\frac 1N\big)\Big(b(\xi_i+\frac 1N)-b(\xi_i-\frac 1N)\big)}{2/N} \quad \text{if}\quad d(x,\xi_i)\leq \frac 1N \text{ for some  } i\in \{ 1,\ldots,k\}. 
        \end{cases}
    \end{equation}

One can verify that $b_N$ is locally Lipschitz continuous, and there exists a constant $K$ not depending on $N$ such that 
\begin{equation}\label {bMbN}
   \sup_{x \in \Xi^{\varepsilon_0}} \Big( |b(x)-b_N(x)| + |b_M(x)-b_N(x)| \Big) \leq K, \text{ for any } M,N >  \frac{1}{\varepsilon_0},
\end{equation}
%\textcolor{red}{Moreover, one-sided Lipschitz continuous coeffcient of $b_N$ does not depend on $N$}. In this proof, denote by the constant $C$ such that
and 
\begin{equation}\label{onesideb}
(x-y)(b_N(x)-b_N(y))\leq |L_1| (x-y)^2, \text{ for any }  x,y \in \mathbb R , N> \frac{1}{\varepsilon_0}.
\end{equation}
From \eqref{bMbN} and Assumption (A3), we have
\begin{equation}\label{boundb1}
    |b_N(x)-b_N(y)|\leq L'_2+2K +3L'_2|x-y|.
\end{equation}
Also, from Assumption  (A1) and \eqref{bMbN},
\begin{equation}\label{boundb2}
    xb_N(x)+\frac{p_0-1}{2}\sigma^2(x)  \leq K \max\{ |\xi_{k}|, |\xi_1|\}+\eta+\gamma x^2.
\end{equation}
By using Theorem 3.1 in \cite{NL17}, the following SDE
$$X_N(t) = x_0 + \int_0^t b_N(X_N(s))ds+ \int_0^t \sigma(X_N(s))dW_s$$
has a unique strong solution.\\
%Recall the function $ \phi_{\delta\varepsilon}$. 
For any $N> M>\frac {1}{\varepsilon_0}$ and $ u\leq T$, we have
$\phi_{\delta\varepsilon}(X_M(u)-X_N(u))=J_1+\frac 12 J_2+J_3$, where 
\begin{align*}
J_1 =  &\displaystyle \int_0^u\phi'_{\delta\varepsilon}(X_M(t)-X_N(t))(b_M(X_M(t))-b_N(X_N(t)))dt,\\
 J_2 = &\displaystyle \int_0^u \phi''_{\delta\varepsilon}(X_M(t)-X_N(t))[\sigma(X_M(t))-\sigma(X_N(t))]^2dt,\\
 J_3 = &\displaystyle \int_0^u \phi'_{\delta\varepsilon}(X_M(t)-X_N(t))[\sigma(X_M(t))-\sigma(X_N(t))]dW_t.
   \end{align*}
   First, by using Assumption  (A4), the Cauchy-Schwarz inequality, and Property (YW5), we have
   \begin{align*}
    \mathbb E[J_{2}]%= &\displaystyle\int_0^T \mathbb E\left[\phi''_{\delta\varepsilon}(X_M(t)-X_N(t)).\big[\sigma(X_M(t))-\sigma(X_N(t))\big]^2\right]dt\\
    \leq \displaystyle\int_0^u  9L_3^2 \mathbb E\left[\frac{2}{ |X_M(t)-X_N(t)| \log \delta }1_{\{|X_M(t)-X_N(t)|\leq \varepsilon\}} |X_M(t)-X_N(t)|\right]dt     \leq   \frac{18L_3^2 T}{\log\delta}.
\end{align*}
Next, we write $J_1 = J_{1,1} + J_{1,2}$, where 
\begin{align*}
    J_{1,1} =& \displaystyle \int_0^u\phi'_{\delta\varepsilon}(X_M(t)-X_N(t))(b_M(X_M(t))-b_M(X_N(t)))dt,\\
    J_{1,2} = & \displaystyle \int_0^u\phi'_{\delta\varepsilon}(X_M(t)-X_N(t))(b_M(X_N(t))-b_N(X_N(t)))dt.
\end{align*}
Using \eqref{onesideb} and Properties (YW1) and (YW2),
\begin{align*}
  J_{1,1}    &=\displaystyle \int_0^u \frac{\phi'_{\delta\varepsilon}(X_M(t)-X_N(t)}{X_M(t)-X_N(t))}(X_M(t)-X_N(t))(b_M(X_M(t))-b_M(X_N(t)))dt\\
    \leq &\displaystyle \int_0^u \frac{\phi'_{\delta\varepsilon}(|X_M(t)-X_N(t)|)}{|X_M(t)-X_N(t))|} |L_1| (X_M(t)-X_N(t))^2dt \leq \displaystyle |L_1| \int_0^u  |X_M(t)-X_N(t)|dt.
\end{align*}
Note that for $N>M$, $b_N(x)=b_M(x)$ when $d(x,\Xi)>\frac 1M$. This implies that 
\begin{align*}
    \mathbb E[J_{1,2}]%=&\displaystyle \int_0^T\mathbb E\Big[\phi'_{\delta\varepsilon}(X_M(t)-X_N(t))(b_M(X_N(t))-b_N(X_N(t)))\Big]dt\\
    &\leq \displaystyle \int_0^u 2K \mathbb E\big(1_{\{X_N(t)\in \Xi^{1/M}\}}\big)dt=\sum_{i=1}^k\displaystyle \int_0^u 2K \mathbb E\big[1_{\{\xi_i-\frac 1M \leq X_N(t)\leq \xi_i+\frac 1M\}} \big]dt. %\leq ck.\frac{1}{M}.
\end{align*}
We will prove that there exists a constant $c$  not depending on $M, N$ such that for all $M,N\geq \frac{1}{\varepsilon_0}$,
\begin{equation}\label{localtimesigmaN}
\int_0^u\mathbb E\big[1_{\{\xi_i-\frac 1M \leq X_N(t)\leq \xi_i+\frac 1M\}} \big]dt \leq 
    \frac cM. 
\end{equation}
Indeed, thanks to Tanaka's formula, for all $a\in\mathbb R$ we have
$$|X_N(u)-a|=|x_0-a|+\displaystyle\int_{0}^u sgn(X_N(t)-a) b_N(X_N(t))ds+\displaystyle \int_0^u sgn(X_N(t)-a) \sigma(X_N(t))dW_t+L_u^a(X_N).$$
Hence
\begin{equation}
|L_u^a(X_N)| \leq |X_N(u)-x_0|+\left|\displaystyle\int_{0}^u sgn(X_N(t)-a)b_N(X_N(t))dt\right|+\left|\displaystyle \int_0^u sgn(X_N(t)-a)\sigma(X_N(t))dW_t\right|.
\end{equation}
By taking expectation on both sides of the above estimate, using Doob's inequality, Proposition \ref{nghiem dung 1}, Assumptions (A3), (A4), the estimates \eqref{boundb2} and \eqref{boundb1}, there exists a constant $c_1$ that does not depend on $N$ and $a$ such that
\begin{align*}
\mathbb E[|L_u^a(X_N)|]  
\leq & \mathbb E\left[\left|\displaystyle\int_{0}^u b_N(X_N(t))dt\right|\right]+\mathbb E\left[\left|\displaystyle \int_0^u \sigma(X_N(t))dW_t\right|\right]\\
&+\mathbb E\left[\left|\displaystyle\int_{0}^u sgn(X_N(t)-a)b_N(X_N(t))dt\right|\right]+\mathbb E\left[\left|\displaystyle \int_0^u sgn(X_N(t)-a)\sigma(X_N(t))dW_t\right|\right]\\
\leq & 2 \displaystyle\int_{0}^u \mathbb E\left[\left|b_N(X_N(t))\right|\right]dt+2 \left[\left|\displaystyle\int_{0}^u \mathbb E\left[\sigma^2(X_N(t))\right]dt\right|\right]^{1/2} \\
\leq & 2 \displaystyle\int_{0}^T \mathbb E\left[\left|b_N(X_N(t))\right|\right]dt+2 \left[\left|\displaystyle\int_{0}^T \mathbb E\left[\sigma^2(X_N(t))\right]dt\right|\right]^{1/2} 
% &\leq 2 \displaystyle\int_{0}^T \mathbb E(|L_2+L_2|X_N(t)|)dt+2 \left[\displaystyle\int_{0}^T \mathbb E\left(|\sigma(0)|+|X_N^{\alpha+1/2}(t)
% \right)|dt\right]^{1/2}.
\leq c_1.
\end{align*}
%where the last estimate is obtained by using Proposition \ref{nghiem dung 1} and Assumptions (A3) and (A4). 
By using the occupation time formula, we obtain
\begin{align}\notag
    \mathbb E\Big[\displaystyle \int_{0}^{u} 1_{[\xi_i-\frac 1M,\xi_i+\frac 1M]}(X_N(t))dt\Big] & \leq \frac{1}{\nu^2} \mathbb E\Big[\displaystyle \int_{0}^{u} 1_{[\xi_i-\frac 1M,\xi_i+\frac 1M]}(X_N(t))\sigma^2(X_N(t))dt\Big]\\\notag
    &=\frac{1}{\nu^2}\displaystyle \int_{-\infty}^{+\infty} 1_{[\xi_i-\frac 1M,\xi_i+\frac 1M]}(a)\mathbb E[L^a_u(X_N)] da\leq \frac{c_1}{\nu^2}\frac 2M = \frac{c}{M}, 
\end{align}
where $c = \frac{2c_1}{\nu^2}.$ Note that  $\mathbb E[J_3]=0$.  In summary, by the Property (YW3),
\begin{align*}
\mathbb E\big[|X_M(u)-X_N(u)| \big]  & \leq\varepsilon+ \mathbb E\big[\phi_{\delta\varepsilon}(X_M(u)-X_N(u))\big] \\
&\leq \varepsilon + \frac{2kKc}{M} +\frac{18L_3^2 T }{\log\delta}+ |L_1| \displaystyle \int_0^u \mathbb E\big[|X_M(t)-X_N(t)|\big]dt.
\end{align*}
By letting $\varepsilon \to 0$ and $\delta \to \infty$, we get 
\begin{align*}
\mathbb E\big[|X_M(u)-X_N(u)| \big] \leq &  \frac{2kKc}{M} + |L_1| \displaystyle \int_0^u \mathbb E\big[|X_M(t)-X_N(t)|\big]dt.
\end{align*}
Using Gronwall's inequality, we have 
\begin{equation*}
   \mathbb E[|X_M(u)-X_N(u)|]\leq  \frac{2kKc}{M}e^{|L_1|u} \text{ for all } N > M > \frac{1}{\varepsilon_0}.
\end{equation*}
Then $X_M(u)$ is a Cauchy sequence in $L^1(\Omega)$, and it converges to a random variable $X(u)$. By letting $N\rightarrow +\infty$, it follows from Fatou's lemma that  
\begin{equation*}
    \mathbb E[|X(u)-X_M(u)|] \leq \frac{2kKc}{M}e^{|L_1|u},\text{ for all } u \leq T.
\end{equation*}
From here, by using Assumption (A4), H\"older's inequality and Young's inequality, we have 
\begin{align}
    &\mathbb E\left[\left(\displaystyle\int_0^u \sigma(X_N(t))dW_t-\displaystyle\int_0^u \sigma(X(t))dW_t\right)^2\right]
    =\mathbb E\left[\displaystyle\int_0^u \left(\sigma(X_N(t)-\sigma(X(t))\right)^2dt\right] \notag \\
    \leq &9L_3^2 \mathbb E\left[\displaystyle\int_0^u |X_N(t)-X(t)|dt\right] \leq \frac{18L_3^2 kKcu}{N} e^{|L_1|u}. \label{Long1} 
\end{align}
Next, from \eqref{localtimesigmaN}, we obtain
$$\mathbb E\left[\displaystyle\int_0^u |b(X_N(t))-b_N(X_N(t)|dt\right]\leq  K\sum\limits_{i=1}^k\mathbb E\Big[\displaystyle \int_{0}^{u} 1_{[\xi_i-\frac 1N,\xi_i+\frac 1N]}(X_N(t))dt\Big]\leq \frac {Kkc}{N}.$$
For $N>M$, we write $\mathbb E\left[\displaystyle\int_0^u |b(X_t)-b(X_N(t))| dt \right]  = I_1 + I_2$, where 
\begin{align*}
 I_1=&\mathbb E\left[\displaystyle\int_0^u |b(X_t)-b(X_N(t))|1_{\{X_t\in \Xi^{1/2M}\}}dt\right],\\
 I_2 =& \mathbb E\left[\displaystyle\int_0^u |b(X_t)-b(X_N(t))|1_{\{X_t \notin \Xi^{1/2M}\}}dt\right].
\end{align*}
From \eqref{localtimesigmaN}, by letting $N\rightarrow \infty$, 
\begin{equation}\notag
\mathbb E\Big[\displaystyle \int_{0}^{u} 1_{[\xi_i-\frac {1}{2M},\xi_i+\frac {1}{2M}]}(X(t))dt\Big]  \leq\limsup\limits_{N\rightarrow\infty}
    \mathbb E\Big[\displaystyle \int_{0}^{u} 1_{(\xi_i-\frac {1}{M},\xi_i+\frac {1}{M})}(X_N(t))dt\Big] \leq \frac{c}{M},
\end{equation}
implying that
\begin{equation}\label{localtimeX}
\mathbb E\Big[\displaystyle \int_{0}^{u} 1_{\{X_t\in\Xi^{1/2M}\}}dt\Big] \leq  \frac{ck}{M}.
\end{equation}
By using the boundedness of $b$ and \eqref{localtimeX}, we have
\begin{equation}\notag
    I_1 \leq 2\|b\|_{\infty}\displaystyle \int_0^u\mathbb E[1_{\{X_t\in\Xi^{1/2M}\}}]\leq 2\|b\|_{\infty}\frac{ck}{M}.
\end{equation}
To estimate $I_2$, we write 
\begin{align*}
    I_2  %= \mathbb E\left[\displaystyle\int_0^T |b(X_t)-b(X_N(t))|1_{\{\{X_t \notin \Theta^{1/M}\}\}}dt\right]\\
    &=\mathbb E\left[\displaystyle\int_0^u |b(X_t)-b(X_N(t))|1_{\{X_t \notin \Xi^{1/2M}\}}1_{\{(X_t,X_N(t))\in S\}}dt\right]\\
    &\quad +\mathbb E\left[\displaystyle\int_0^u |b(X_t)-b(X_N(t))|1_{\{X_t \notin \Xi^{1/2M}\}}1_{\{(X_t,X_N(t))\notin S\}}dt\right]\\
   & \leq 3L_2 \displaystyle \int_0^u\mathbb E[|X(t)-X_N(t)|]dt+2\|b\|_{\infty} \displaystyle \int_0^u \mathbb E[|1_{\{|X_t-X_N(t)|>1/2M\}}]dt\\
   & \leq 3L_2 \displaystyle \int_0^u\mathbb E[|X(t)-X_N(t)|]dt+4\|b\|_{\infty}M \displaystyle \int_0^u \mathbb E[|X_t-X_N(t)|]dt\\
   %&\leq 3L_2CT^2.e^{CT}.\frac 1N+\|b\|_{\infty}.CT^2.e^{CT}.\frac{M}{N}
    & \leq (3L_2 + 4M\|b\|_\infty) \frac{2kKc}{N}e^{|L_1|u}.
\end{align*}
By choosing $N=M^2\rightarrow \infty$, we have 
$$\lim_{N \to \infty} \mathbb E\left[\displaystyle\int_0^u |b(X(t))-b_N(X_N(t)|dt\right] = 0.$$
Together with \eqref{Long1}, this concludes the existence of solution to equation \eqref{eqn1.1}. The uniqueness of solution is obtained by using Lemma \ref{uniquess}.
\end{proof}
Now we are ready to prove Theorem \ref{existence2}. For any $N>\max\{|\xi_1|,|\xi_k|\}$, let's denote 
\begin{equation*}
    b_N(x)=
    \begin{cases}
        b(x) & \text{ if }  |x|\leq N,\\
        b\Big(\frac{Nx}{|x|}\Big)(N+1-|x|\Big)  & \text{ if }  N<|x|<N+1,\\
        0  & \text{ if }  |x|\geq N+1,
    \end{cases}
\end{equation*}
and
\begin{equation*}
    \sigma_N(x)=
    \begin{cases}
        \sigma(x)  & \text{ if } |x|\leq N,\\
        \sigma\Big(\frac{Nx}{|x|}\Big)(N+1-|x|)   & \text{ if }  N<|x|<N+1,\\
        0  & \text{ if }  |x|\geq N+1.
    \end{cases}
\end{equation*}
One can check that $b_N$ and $\sigma_N$ satisfy all conditions of Lemma \ref{existence2.1}, with the note that a compactly supported $(\alpha+1/2)$-H\"older continuous function is also a $1/2$-H\"older continuous function.
Then the equation 
\begin{equation*}
		\begin{cases} 
			& d X_N(t)=b_N(X_t)dt+\sigma_N(X_t)dW_t \\
			& X^N_0=x_0 \in \mathbb R
		\end{cases},
	\end{equation*}
 has a unique strong solution.
 Moreover, we can verify that $xb_N(x)+\frac{p_0-1}{2}|\sigma_N(x)^2|\leq \gamma'|x|^2+\eta'$ for some constant $\gamma',\eta'$ which depend only on  $\gamma,\eta$.
Hence, it follows from Lemma 3.1 in \cite{NL19} that if $p_0\geq (l\vee m)+4$ and $2\leq p\leq (p_0-(l\vee m))/2$, then
$$\mathbb E\Big[ \sup\limits_{t\in [0,T]}|X^N_t|^p\Big] \leq C(x,T,l,p,\gamma,\eta,m).$$
% The existence of solution hold from the argument of Theorem 3.1,ii in [LT17].
Again, by following the argument in the proof of Theorem 3.1.i in \cite{NL17}, we can show that  $X_N$ will converge in probability to a process $X$ which satisfies equation \eqref{eqn1.1}. 

The uniqueness of solution is obtained by using Lemma \ref{uniquess}. We conclude the proof.

\section{Numerical analysis}
\label{Sec:4} 
In this section, we conduct the proposed algorithm for several SDEs to support our theory. Recall that if $X$ is the unique strong solution of the SDE \eqref{eqn1.1} and $Y^\Delta$ is our approximation scheme corresponding to the step-size parameter $\Delta$, then under the conditions stated in Theorem \ref{convergence} and Theorem \ref{convergence2}, for $\alpha \in (0,1/2],$ we have
\begin{equation*} 
	\sup\limits_{0\leq t\leq T} \bE \left[  |Y_t^\Delta - X_t| \right] \le 
	C \Delta^{\alpha} 
	\end{equation*}
for some $C\in(0, \infty)$ and for all $\Delta\in(0,\Delta_0)$. Hence, by using the triangle inequality, we obtain
\begin{equation*}
\sup\limits_{0\leq t\leq T} \bE \bigl[ |{Y_t}^{2^{-(k+1)}\Delta_0} - {Y_t}^{2^{-k}\Delta_0}|\bigr]\leq C 2^{- k\alpha}
\end{equation*}
for some positive constant $C$.

If we denote $r_k := |{Y_t}^{2^{-(k+1)}\Delta_0} - {Y_t}^{2^{-k}\Delta_0}|$, then we can use $\widetilde r_k := \frac{1}{N_k}\sum_{n=1}^{N_k}r_{k,n}$ as an approximation for $\bE[r_k]$, in which $r_{k,n}$ is the $n^{th}$ numerical path of $r_k$ for all $n = 1, 2,..., N_k$.
The main issue here is that the simulations of the fine and coarse paths in the formula of $r_k$ have to share the same driving Brownian motion during the overlap time interval. To do this, we adapt a simulating method from Fang and Giles (see \cite{FG}).

To determine the convergence rate $\alpha$ for each numerical example, observe that since
\begin{equation*}
    \widetilde r_k \leq c  2^{- k \alpha}
\end{equation*} 
when the approximation $\widetilde r_k$ is relatively good, i.e.\ when $N_k$ is sufficiently large, we also have
\begin{equation*}
    \log\widetilde r_k \leq \log c -  \alpha(k \log 2 ).
\end{equation*}
Therefore, we can estimate the rate $\alpha$  by computing the slope of the regression line for the pairs $(-k \log 2, \log\widetilde r_k)$, i.e. it should be no more than an empirical rate $\Tilde{\alpha}$ which satisfies $- \log\widetilde r_k = -\log c + \Tilde{\alpha}(k \log 2 ) + o(1).$

In Theorem \ref{cost analysis}, we have provided the rate of the computation cost $\zeta$ satisfying $\bE[N_T] \leq CT\Delta^\zeta$ for our scheme. This theoretical rate can also be estimated by an empirical rate $\Tilde{\zeta}$ in the same way as above.

The two SDEs that we consider in Table \ref{tab:1} satisfy the conditions in Theorem \ref{convergence}, while those in Table \ref{tab:2} satisfy the conditions provided in Theorem \ref{convergence2}. For each example, we compute $\widetilde r_k$ for $k = 1,2,3,4,5$ with $\Delta_0 = 1.8\cdot10^{-4}$. The number of samples is $N_k = 10^3$ for $k = 1,2,3,4$. The results for the case $T=1$ are summarized in Table \ref{tab:3}.

\begin{center}
	\begin{table}[!ht]
		\begin{center}
			\begin{tabular}{|c|c|c|c|c|c|c|c|c|c|}
				\hline
				Example
				&
				$b(x)$
				&
			$\sigma(x)$
				&
			 $x_0$
    &
    $p_0$
            &
            $l$
            &
            $m$
            &
            $\alpha$
            &
            $\gamma$
            &
            $L_1$
				\\
				\hline
				1
				&
				$
		\begin{cases} 
			-x-x^3 &\text{ if } x\geq0, \\
 		1- x - x^3 &\text{ if } x<0. \\  
		\end{cases}
$
				&
				$
		(1+x)(1+x^{2/3})1_{[-1,\infty)}(x)
$
				&
			 0
            &
            20
            &
            2
            &
            1
            &
            $\dfrac{1}{6}$
            &
            -1
            &
            -1
				\\
    \hline
    				2
				&
	$			\begin{cases} 
			-1+x-x^3 &\text{ if } x\geq0, \\
 		x - x^3 &\text{ if } x<0. \\  
		\end{cases}
  $
				&
					$
		(1+x)(1+x^{2/3})1_{[-1,\infty)}(x)
$
				&
			 0
            &
            20
            &
            2
            &
            1
            &
            $\dfrac{1}{6}$
            &
            1
            &
            1
				\\
    \hline
			\end{tabular}
		\end{center}
		\caption{\small Examples of SDEs satisfying the conditions in Theorem \ref{convergence}.
			\label{tab:1}}
	\end{table}
\end{center}

\begin{center}
	\begin{table}[!ht]
		\begin{center}
			\begin{tabular}{|c|c|c|c|c|c|c|c|c|c|}
				\hline
				Example
				&
				$b(x)$
				&
			$\sigma(x)$
				&
			 $x_0$
                &
              $p_0$
            &
            $l$
            &
            $m$
            &
            $\alpha$
            &
            $\gamma$
            &
            $L_2$
				\\
				\hline
    		3
				&
$
		\begin{cases} 
			1+x-x^3 &\text{ if } x>2, \\
			x^2+1 &\text{ if } 0 \leq x \leq 2, \\
 		x - x^3 &\text{ if } x<0. \\  
		\end{cases}
$
				&
$
		1+\sqrt{\dfrac{x^4+x^{4/3}}{14}}
$
				&
			 0.2
                &
              26
            &
            2
            &
            1
            &
            $\dfrac{1}{6}$
            &
           -1
            &
            -1
				\\
				\hline
			4
				&
$
		\begin{cases} 
			1+x-x^{2/3} &\text{ if } x>2, \\
			x^2+1 &\text{ if } 0 \leq x \leq 2, \\
 		x &\text{ if } x<0. \\  
		\end{cases}
$
				&
			$1+x^{2/3}$
				&
			 0
                &
              20
            &
            1
            &
            $\dfrac{4}{3}$
            &
            $\dfrac{1}{6}$
            &
            1
            &
            1
				\\
				\hline
			\end{tabular}
		\end{center}
		\caption{\small Examples of SDEs satisfying the conditions in Theorem \ref{convergence2}.
			\label{tab:2}}
	\end{table}
\end{center}

\begin{center}
	\begin{table}[!ht]
		\begin{center}
			\begin{tabular}{|c|c|c|c|c|}
				\hline
				Example
				&
	$\Tilde{\alpha}$
				&
			$95\%$ CI for $\Tilde{\alpha}$
				&
    $\Tilde{\nu}$
			 	&
			$95\%$ CI for $\Tilde{\zeta}$
				\\
				\hline
				1
				&
				0.557
				&
			$[0.422, 0.692]$
				&
			 -1.131
			 &
			$[-1.286, -0.977]$
				\\	
             \hline
				2
				&
				0.534	

				&
			$[0.376, 0.692]$
				&
			 -1.139	
			 &
			$[-1.279, -0.999]$
				\\	
            \hline
				3
				&
				0.601
				&
			$[0.463, 0.739]$
				&
			 -1.252
			 &
			$[-1.375, -1.129]$
				\\
				\hline
    				3
				&
				0.608

				&
			$	 [0.485, 0.731]$
				&
			 	-1.094	
			 &
			$[-1.137, -1.052]$
				\\
				\hline
			\end{tabular}
		\end{center}
		\caption{\small Estimation for the rates $\alpha$ and $\zeta$ in Example 1,2, 3 and 4 for $T = 1$. 
			\label{tab:3}}
	\end{table}
\end{center}
In these examples, for $T=1$, the empirical rate of convergence is larger than the theoretical rate $1/6$,  and the empirical rate of computation cost is almost the same as the theoretical rate $-1$. Hence, the numerical results support the theoretical findings. 

In addition, since $\gamma < 0$ in Example 1 and 3, the constant $C$ in the estimates \eqref{EXmu-X1} and \eqref{EXmu-X2} does not depend on $T$. Hence, the intercept of the regression line for the pairs $(-k \log 2, \log\widetilde r_k)$ should not change drastically when the time $T$ is adjusted. This can be verified when we compare the values of this intercept in the case $T=1$ and $T=5$ for both examples. The results are shown in Table \ref{tab:4}.

\begin{center}
	\begin{table}[!ht]
		\begin{center}
			\begin{tabular}{|c|c|c|}
				\hline
                Example
				&
				$T=1$
				&
			$T=5$
   \\
\hline
                1
				&
				-4.390
				&
			-4.184
   \\
\hline
                3
				&
				-1.925
				&
			-1.878
   \\
\hline
			\end{tabular}
		\end{center}
		\caption{\small The intercepts of the regression line for the pairs $(-k \log 2, \log\widetilde r_k)$ in Example 1 and 3. 
			\label{tab:4}}
	\end{table}
\end{center}

It has been shown that when $\gamma < 0$, the constant $C$ in Theorem \ref{cost analysis} does not depend on $T$ either. To demonstrate this, we note that the intercept of the regression line in the estimation of $\zeta$ is $\log C + \log T$. The values of this intercept for Example 1 in the case $T=1$ and $T=5$ are 7.082 and 7.675, respectively. We can see that $7.645 - 7.082 \approx \log 5$, which suggests that the constant $C$ does not depend on $T$.

%\section*{Acknowledgements}
%This research is funded by the Vietnam National Foundation for Science and Technology
%Development (NAFOSTED) under grant number  101.03-2021.36. 

\end{document}